\documentclass[a4paper,12pt]{article}
\usepackage{a4wide}

\usepackage{amsmath,amssymb}
\usepackage{theorem}
\usepackage{xspace}
\usepackage{url}
\usepackage{bm}
\usepackage{lscape}
\usepackage{soul}
\usepackage{color,xcolor}
\usepackage{graphicx}
\usepackage{subfigure}
\usepackage{enumitem}
\usepackage{listings}
\usepackage{etoolbox}
\usepackage{float}

\usepackage{ulem}

\numberwithin{equation}{section}

\definecolor{ivogreen}{HTML}{009538}

\newtheorem{theorem}{Theorem}
\newtheorem{remark}{Remark}[section]

\theorembodyfont{\normalfont}

\newtheorem{definition}{Definition}[section]

\newcommand{\cA}{{\mathcal A}}

\newcommand{\cP}{{\mathcal P}}

\newcommand{\rref}[1]{(\ref{#1})}

\newcommand{\R}{\mathbb{R}}
\newcommand{\e}{\textup{e}}
\renewcommand{\i}{\textup{i}}

\newcommand{\bfg}{\mathbf g}

\newcommand{\bfn}{\mathbf n}

\newcommand{\bft}{\mathbf t}

\newcommand{\bfuno}{\mathbf 1}

\newcommand{\bfnn}{{\boldsymbol n}}\newcommand{\bfff}{{\boldsymbol f}}\newcommand{\bfii}{{\boldsymbol i}}\newcommand{\bfjj}{{\boldsymbol j}}\newcommand{\bfkk}{{\boldsymbol k}}
\newcommand{\bftheta}{{\boldsymbol\theta}}

\title{Fine spectral analysis of preconditioned matrices and matrix-sequences arising from stage-parallel implicit Runge-Kutta methods of arbitrarily high order}
\author{Ivo Dravins\thanks{Department of Information Technology,  Uppsala University, Sweden ({ivo.dravins@it.uu.se})} 
\and Stefano Serra--Capizzano\thanks{\mbox{Department of Humanities and Innovation, University of Insubria, Italy ({s.serracapizzano@uninsubria.it})}}
\and Maya Neytcheva\thanks{Department of Information Technology,  Uppsala University, Sweden ({maya.neytcheva@it.uu.se})}}

\begin{document}

\lstset{%
  numbers=none,
  tabsize=3,
  breaklines=true,
  basicstyle=\small\ttfamily,
  framerule=0pt,
  backgroundcolor=\color{gray!25},
  columns=fullflexible
}

\maketitle

\begin{abstract}
The use of high order fully implicit Runge-Kutta methods is of significant importance in the context of the numerical solution of transient partial differential equations, in particular when solving large scale problems due to fine space resolution with many millions of spatial degrees of freedom and long time intervals. In this study we consider strongly A-stable implicit Runge-Kutta methods of arbitrary order of accuracy, based on Radau quadratures, for which efficient preconditioners have been introduced. A refined spectral analysis of the corresponding matrices and matrix-sequences is presented, both in terms of localization and asymptotic global distribution of the eigenvalues. Specific expressions of the eigenvectors are also obtained. The given study fully agrees with the numerically observed spectral behavior and substantially improves the theoretical studies done in this direction so far. Concluding remarks and open problems end the current work, with specific attention to the potential generalizations of the hereby suggested general approach.
\end{abstract}



\section{Introduction}\label{sec0}
Runge--Kutta methods constitute a widely used class of time integration methods for solving (systems of) ordinary differential equations (ODEs). Consider a system of ODEs of the form
\[
\frac{\partial \bm{u}}{\partial t} = f(\bm{u},t)
\]
with some initial conditions.
Here $\bm{u}\in \R^{n}$ is the unknown vector function.

In the general framework of the Runge-Kutta methods the solution at the next time-step is then approximated by
\begin{equation}\label{eq_IRK1}
    \bm{u}_{n+1} = \bm{u}_{n} + \tau \sum_{i=1}^q b_i \bm{k}_i,
\end{equation}
where $\tau$ is the time step and $\bm{k}_i, \ i=1,\hdots,q$, denote intermediate variables referred to as \textit{stages}. The stage variables are implicitly defined by the relations
\begin{equation}\label{eq_IRK2}
    \bm{k}_i = f(\bm{u}_n + \tau \sum_{j=1}^q a_{ij} \bm{k}_j ,t_n + \tau c_i ).
\end{equation}
The method is characterized by the Runge-Kutta matrix $A_q$ and two vectors $\bm{b}, \bm{c}$ often written together in a table, referred to as the  \textit{Butcher tableau} as follows,
\begin{table}[H]
\centering
\begin{tabular}{l|llll}
$c_1$    & $a_{11}$ &  $a_{12}$    &  $\hdots$        &     $a_{1q}$     \\
$c_2$    & $a_{21}$ & $a_{22}$ &   $\hdots$       &       $a_{2q}$    \\
$\vdots$ & $\vdots$ & $\vdots$ & $\ddots$ &     $\vdots$      \\
$c_q$    & $a_{q1}$ & $a_{q2}$ & $\hdots$ & $a_{qq}$ \\ \hline
         & $b_1$    & $b_2$    & $\hdots$ & $b_q$
\end{tabular}
= \begin{tabular}{l|l}
${\bf c}$&$A_q$\\\hline &${\bf b}$
\end{tabular}.
\end{table}

The Runge-Kutta (RK) methods can be broadly classified by the structure of the matrix $A_q$, namely,  we have explicit RK methods when $A_q$ is strictly lower-triangular, diagonally implicit Runge-Kutta (DIRK) methods when $A_q$ is lower-triangular with nonzero diagonal and fully implicit RK methods (IRK) when $A_q$ is generally dense.

The IRK methods are characterized by their high order time-discretization and possessing strong stability properties. Despite these desirable features IRK methods have seen a relatively limited use, mainly due to the high computational cost when being implemented. Indeed, we see from \rref{eq_IRK1} that each time-step necessitates the solution of the linear system \rref{eq_IRK2} of dimension $qn$, where $n$ denotes the number of equations in the considered system of ODEs and can be very large and $q$ is the number of stages in the IRK method.

In this work we consider only linear problems, thus, ODE systems of the form
\begin{equation}\label{eq1.1}
 	 M \frac{\partial \bm{u}(t)}{\partial t} + K \bm{u}(t) = \bm{f}(t),
 \end{equation}
 where $M$ and $K$ are square matrices in $\R^{n\times n}$ and $\bm{u}(t)$, $\bm{f}(t)$ are vectors in $\R^{n}$.
We assume that these arise after semi-discretization in space of some non-stationary linear partial differential equation of the type
\begin{equation}\label{eq_pde}
\displaystyle{\frac{\partial u}{\partial t}} - \varepsilon\Delta u +\bm{b}\cdot\nabla u = f(t),
\end{equation}
equipped with appropriate initial and boundary conditions.
Further, we use the IRK method based on the Radau (or Gauss-Radau) integration method, known also as Radau IIA. It has approximation order $2q-1$, where $q$ is the number of stages. In addition, it is a \textit{strongly} $A$-stable (also called \textit{strongly} $L$-stable) method, for a definition, see for instance \cite{Ax64}.
An additional advantage is that the Radau method is shown to be \textit{stiffly accurate}, cf. e.g., \cite{HairerWanner_1991}), thus, it does not exhibit  order reduction, that can occur for instance when solving systems of differential-algebraic equations, see e.g., \cite{Petzold,AxBla}.

The fully discrete analogue of \rref{eq1.1} in Kronecker product (tensor) form reads
\begin{equation}\label{base_kronfrom}
     \cA_0\bm{k} = (\mathbb{I}_q\otimes M + \tau A_q\otimes K) \bm{v} = \overline{\bm{g}} - (\mathbb{I}_q\otimes K) (\bm{e}_q\otimes \bm{u}_0),
\end{equation}
or, as $A_q^{-1}$ is nonsingular, in transformed form
\begin{equation}\label{trans_kronfrom}
     \cA\bm{k}=( A_q^{-1}  \otimes M + \tau \mathbb{I}_q \otimes K) \bm{k} = (A_q^{-1} \otimes \mathbb{I}_n) \overline{\bm{g}} - (A_q^{-1} \otimes K)(\bm{e}_q \otimes \bm{u}_0 ).
\end{equation}
Here $\tau$ is the time step.

In this study we consider one particular preconditioner, $\cP$, for the matrix $\cA$ in the linear system \rref{trans_kronfrom}, as arising from the Radau IIA-type IRK discretizations of \rref{eq1.1}. It enables only real arithmetic and allows for stage-parallel implementation. Specifically, we are concerned with the task to analyse and estimate the spectrum of the preconditioned matrix $\cP^{-1}\cA$ in order to explain the high numerical efficiency and the fast convergence behaviour observed in the numerical results in \cite{IRKtheory} and in \cite{stageparallel}.

This work significantly improves the previous eigenvalue bound for the two stage case from \cite{IRKtheory}. We thoroughly analyse the case $q=2$, further show the detailed spectral behavior for the three stage case and present a general machinery in which we, for a general $q$-stage method, construct polynomials of degree $q-1$, whose roots determine the eigenvalues of the preconditioned system.
In fact, both in terms of localization and global distribution of the eigenvalues of the preconditioned system, our results are tight and cannot be improved theoretically, where the latter statement is justified in a rigorous way and also observed in a wide set of numerical tests.

The paper is structured as follows.  A description of the arising matrices and the proposed preconditioner is given in Section \ref{sec1}. In Section \ref{sec:sp_tools} we describe the theoretical tools used in the analysis of the spectral properties of the preconditioned matrix.
The spectral analysis itself is presented in Sections \ref{sec:sp_loca} and
Section \ref{sec:sp_distr}.
Section \ref{sec:numerics} contains the numerical test suite.
Conclusions and possibilities for future extensions of the results are found in Section \ref{sec:final}.

\section{Linear systems arising in IRK and preconditioning }\label{sec1}

We consider the semi-discrete system of ODEs from \rref{eq1.1}, arising after the spatial discretization of a  partial differential equation as in \rref{eq_pde}. We assume that the space discretization is done by some suitable finite element method (FEM), thus,  $M$ is a mass matrix, $K$ is a stiffness matrix. The matrix $M$ is symmetric and positive definite by construction. In this study we assume that $K$ possesses the same property, that is ensured when $\bm{b}$ is zero. However, the technique that we introduce can handle also the general case and this is discussed in the conclusions.

The linear system that has to be solved in order to determine the stage variables
is either of the form \rref{base_kronfrom} or as in the transformed form \rref{trans_kronfrom}.

Next we make a choice between $\cA_0$ and $\cA$ in favour of the latter. As discussed in \cite{IRKtheory}, constructing a stage-parallel preconditioner  to $\cA_0$ would require solutions with the diagonal blocks of $I_q\otimes M+\tau(A_q\otimes K)$ while the preconditioner to $\cA$ requires solution of systems with the diagonal blocks of $A_q^{-1}\otimes M+\tau(I_q\otimes K)$. (Here and in the sequel, $I_m$ denotes the identity matrix of order $m$.)
When $K$ is ill-conditioned then the block matrices in the former case are also ill-conditioned. To reduce the ill-conditioning we choose the latter form, which also facilitates finding a good preconditioner of these matrices too, for instance of algebraic multilevel (AMLI) type or of algebraic multigrid (AMG) type method \cite{Notay,Vassilevski}.

The detailed form of system to be solved to determine the stage variables, as well as that of the matrix $\cA$ is given by
\begin{equation}
\cA\begin{bmatrix}\bm{k}_1 \\ \bm{k}_2 \\ \vdots \\ \bm{k}_q\end{bmatrix} \equiv
\begin{bmatrix}
a_{11}M +  \tau  K &  a_{12}M    &  \hdots        &    \tau a_{1q}M     \\
  a_{21}M & M + a_{22}M+\tau K &   \hdots       &     \tau  a_{2q}M    \\
 \vdots & \vdots & \ddots &     \vdots      \\
 a_{q1}M &a_{q2}M & \hdots  & a_{qq}M + \tau K
\end{bmatrix}
\begin{bmatrix}\bm{k}_1 \\ \bm{k}_2 \\ \vdots \\ \bm{k}_q\end{bmatrix}
= \begin{bmatrix}
\bm{f}_1  \\
\bm{f}_2  \\
\vdots \\
\bm{f}_q
\end{bmatrix}
\end{equation}
As argued above, we work with $\cA$ as arising in
the class of the Radau IIA IRK methods to take full advantage of their attractive accuracy and stability properties, cf. e.g.,  \cite{Ax64,HairerWanner_1991,Petzold,AxBla,Lambert,Ehle_1973}.

In addition to the above qualities, the IRK matrices, arising from Radau IIA quadratures, possess some special algebraic properties, used when constructing a preconditioner to $\cA$, which makes the preconditioner superior to other preconditioners of similar algebraic structure, see \cite{Rana} and the discussion in \cite{IRKtheory}.

Our analysis utilizes the rational form of the RK matrices $A_q$ and we show in Figure \ref{fig_Aq_rat} two such matrices as an illustration.

\begin{figure}[H]
\centering
\[
\renewcommand\arraystretch{1.5}
{ \everymath={\displaystyle}
\begin{array}{c|cc}
\frac{1}{3}    & \frac{5}{12} &  -\frac{1}{12} \\[8pt]
1    & \frac{3}{4} & \frac{1}{4} \\[8pt] \hline
 \rule{0pt}{20pt}  & \frac{3}{4}    & \frac{1}{4}
\end{array}} \qquad
{ \everymath={\displaystyle}
\begin{array}
{c|ccc}
\frac{2}{5}- \frac{\sqrt{6}}{10}    & \frac{11}{45}-\frac{7\sqrt{6}}{360} &  \frac{37}{225}- \frac{169\sqrt{6}}{1800}  & -\frac{2}{225}+\frac{\sqrt{6}}{75}     \\ [8pt]
\frac{2}{5}+ \frac{\sqrt{6}}{10}    &  \frac{37}{225} + \frac{169\sqrt{6}}{1800} &  \frac{11}{45}+\frac{7\sqrt{6}}{360}       &       -\frac{2}{225}-\frac{\sqrt{6}}{75}    \\[8pt]
1 & \frac{4}{9}-\frac{\sqrt{6}}{36} & \frac{4}{9}+\frac{\sqrt{6}}{36} &  \frac{1}{9} \\[8pt] \hline
  \rule{0pt}{20pt}  & \frac{4}{9}-\frac{\sqrt{6}}{36} & \frac{4}{9}+\frac{\sqrt{6}}{36} &  \frac{1}{9}
\end{array}}
\]
\caption{Examples of Butcher tableau, $q=2$ (left), the $q=3$ (right)}\label{fig_Aq_rat}
\end{figure}

\subsection*{Preconditioning}
As the system of equations \eqref{trans_kronfrom} is of dimension of $qn$, where $q$ is the number of stages and $n$ is the number of spacial degrees of freedom, it can be very large, necessitating the use of iterative solution methods such as GMRES or GCR (see \cite{Ax-book,Saad-book} and references therein) combined with some efficient preconditioning technique.
We aim at constructing a preconditioner, which is both numerically efficient, i.e., resulting in tight clustering of the eigenvalues of $\cP^{-1}\cA$, as well as computationally efficient, which in this case includes stage parallelism. We also pose the requirement, when applying the preconditioner, to use only real arithmetic.

The preconditioner is based on some derivations in \cite{Ax64}, showing that the entries in the lower-triangular part of the matrix $A_q$ are by value larger than those in the strictly upper triangular part, thus, the lower-triangular part is expected to be a good approximation of $A_q$. The property is inherited by $A_q^{-1}$.

A preconditioner based on the lower-triangular factor of a particular $LU-$decom- position of $A_q^{-1}$ is first proposed in \cite{AxNey_Algoritmy}, namely,
let factorize $A_q^{-1}=L_qU_q$, where $U_q$ has unit diagonal. Because of the dominating property of $L_q$, $\|U_q-I_q\|$ is small, in particular less than $1$.
The lower-triangular factor $L_q$ is real-valued diagonalizable and its spectral decomposition  $L_q = T_q \Lambda_q T_q^{-1}$ is easily computed. The matrix $\Lambda_q$ contains the diagonal entries of $L_q$ and the matrices $T_q$ are also of lower-triangular form and can be computed by a simple recursion.
The preconditioner is of the form
\begin{equation}\label{eq_precPL}
\cP=L_q\otimes M + \tau I_q\otimes K.
\end{equation}
Note, that the spectral decomposition of $L_q$ enables parallelization across the stages while avoiding complex arithmetic. Indeed, we see that
\begin{equation}\label{eq_transdiag}
\cP = L_q\otimes M + \tau I_q\otimes K =(T_q\otimes I_n)\underbrace{\left( \Lambda_q\otimes M + \tau I_q\otimes K\right)}_{\cP_d}(T_q^{-1}\otimes I_n),
\end{equation}
where $\cP_d$ is block diagonal with $q$ blocks of size $n$. The form \rref{eq_transdiag} allows for the action of $\cP_d^{-1}$ to be computed in a stage-parallel fashion, and as shown in \cite{stageparallel}, the cost of the $T$-transformations is small. As a side note, we mention that the parallel behavior of the preconditioner and comparisons between the stage parallel and stage serial versions is studied in \cite{stageparallel}, however, it falls out of the scope of the current study and is not considered any further.

Clearly, all computations when applying $\cP$ require real arithmetic. This is in contrast to the idea to use the spectral decomposition of $A_q$, cf. e.g., \cite{Butcher_1976}, which entails complex arithmetic because some of the eigenvalues of $A_q$ appear in complex-conjugate pairs.

The spectral properties of $\cP^{-1}\cA$ are studied in \cite{IRKtheory} and a conjecture regarding the distribution of the eigenvalues of the preconditioned system is made, combined with a rigorous derivation of a spectral bound only for the two-stage case.
The current work focuses on analysis of the spectrum of the preconditioned system. For this, the above spectral decomposition is not needed but is nonetheless mentioned as it is an important implementation-related detail.

\section{Theoretical and spectral tools for matrix analysis}\label{sec:sp_tools}
In this section we present the main analysis tools that play a crucial role in part of the derivations in Section \ref{sec:sp_loca} and in the whole study in Section \ref{sec:sp_distr}. In particular, we recall the concept of (multilevel) Toeplitz matrices and that of the related matrix-sequences, of preconditioned Toeplitz structures and of the associated preconditioned matrix-sequences, and that of spectral distribution in the Weyl sense (see, for example, \cite{GLT-block1D,GLT-blockdD,GLT1_book,GLT2_book} for a complete account of the relevant theory).

\subsection{Multilevel block Toeplitz matrices, preconditioned structures, and spectral distribution}\label{ssec:toeplitz-distribution}
Toeplitz matrices are a particular class of matrices, characterized by the fact that all their diagonals parallel to the main one have constant values. Namely, we write $ \left( T_n \right)_{i,j} = t_{i-j} $, $ i,j = 1,...,n $, to denote a Toeplitz matrix of size $ n $, where $ t_k $ is a constant for every $ k = 1-n,...,n-1 $.\\
When each $ t_k $ is a square matrix of fixed dimension $r$ we say that $ T_n $ is a $r$-block Toeplitz matrix.\\
Then, in a recursive manner, it is possible to define a $d$-level Toeplitz matrix as follows: a $d$-level Toeplitz matrix is a Toeplitz matrix where each ``coefficient" $t_k$ denotes a $ (d-1) $-level Toeplitz matrix. Namely, using a standard multi-index notation, we can write a $ d$-level Toeplitz matrix as
\begin{equation*}
T_{\bfnn} = \left( t_{\bfii-\bfjj} \right)_{\bfii,\bfjj=\bfuno}^{\bfnn} \in \mathbb{C}^{N(\bm{n}) \times N(\bm{n})},
\ \ \ N(\bm{n})=n_1\cdots n_d,
\end{equation*}
where $ \bfnn = \left( n_1,...,n_d \right) $ is a positive integer multi-index (i.e. $ 0 < n_i \in \mathbb{N} $ for every $ i $) and $ t_{\bfkk} \in \mathbb{C} $ for every $ \bfkk = -\left(\bfnn - \bfuno\right),..., \bfnn - \bfuno $, $ \bfuno $ denoting the vector in $ \mathbb{Z}^d  $ of all ones. When the basic elements $ t_\bfkk \in \mathbb{C}^{r \times r} $ for some $ 1 \leq r\in \mathbb{N} $, we say that $ T_\bfnn $ is a $ d$-level $ r $-block Toeplitz matrix.\\
We are particularly interested in the case where the matrix $ T_\bfnn $ is  generated by a function $ \bfff \in L^1([-\pi,\pi]^d) $. Namely, given a function $ \bfff : [-\pi,\pi]^d \to \mathbb{C}^{r \times r} $ in $ L^1([-\pi,\pi]^d) $ we denote its Fourier coefficients as
\begin{equation*}
\hat{\bfff}_\bfkk=\frac1{(2\pi)^d}\int_{[-\pi,\pi]^d}\bfff(\bftheta)\e^{-\i\,\bfkk\cdot\bftheta}d\bftheta\in\mathbb C^{r \times r},\quad\bfkk\in\mathbb Z^d, \quad \bfkk\cdot\bftheta=\sum_{i=1}^d k_i \theta_i
\end{equation*}
and define the associated sequence of $ d $-level $ r $-block Toeplitz matrices by
\begin{equation*}
 \left\{ T_{\bfnn,r}(\bfff) \right\}_\bfnn, \quad \  T_{\bfnn,r}(\bfff):=\left(\hat{\bfff}_{\bfii-\bfjj}\right)_{\bfii,\bfjj=\bfuno}^{\bfnn}\in \mathbb{C}^{rN(\bm{n}) \times r N(\bm{n})},\qquad\bfnn\in\mathbb N^d.
\end{equation*}

The multi-index $\bfjj$ has to be understood as $(j_1,\ldots,j_d)$ and the ordering is lexicographical as in, 
for instance, \cite{ty-1} or in the books \cite{GLT-blockdD,GLT2_book}, with $\bfjj < \bfkk$ if $j_l\le k_l$, for all $l=1,\ldots,d$, and $\bfjj\neq \bfkk$.

\noindent
For a square matrix $ X_{\bfnn,r} $ of dimension $ d_\bfnn $, where $ r $ is a constant independent of $ \bfnn $, define
$$
\Sigma_\sigma(F,X_{\bfnn,r}):=\frac{1}{d_\bfnn}\sum_{k=1}^{d_\bfnn}F(\sigma_k(X_{\bfnn,r})), \quad \Sigma_\lambda(F,X_{\bfnn,r}):=\frac{1}{d_\bfnn}\sum_{k=1}^{d_\bfnn}F(\lambda_k(X_{\bfnn,r})),
$$
where $\sigma_k(X_{\bfnn,r})$ and $\lambda_k(X_{\bfnn,r})$ denote the singular values and the eigenvalues of $X_{\bfnn,r}$, respectively, sorted in non-decreasing order.\\
Hereafter, the symbol $ \left\{ X_{\bfnn,r} \right\}_\bfnn $ is used to denote a sequence of matrices of increasing dimension $ d_\bfnn $ such that $ d_\bfnn \to \infty $ as $ \bfnn \to \infty $,  the notation $ \bfnn \to \infty $ means that $ n_i \to \infty $ for every $ i=1,...,d $.

\begin{definition}[Spectral symbol]\label{def:spectral-distr}
Let $ X_{\bfnn,r} $ be a matrix-sequence and let $ \bfff : D \to \mathbb{C}^{r \times r} $ be a Hermitian matrix-valued measurable function defined on a measurable set $ D \subset \mathbb{R}^m $ such that $ 0 < \mu_m(D) < \infty $, where $ \mu_m $ denotes the Lebesgue measure on $ \mathbb{R}^m $. We say that $ \left\{ X_{\bfnn,r} \right\}_\bfnn $ is distributed like $ \bfff $ in the sense of eigenvalues, if for every $ F \in C_c(\mathbb{R}) $, we have
\begin{equation*}
\lim_{\bfnn \to \infty}  \Sigma_\lambda(F,X_{\bfnn,r}) = \frac{1}{\mu_m(D)} \int_{D} \frac{1}{r}\sum_{k=1}^{r} F(\lambda_k(\bfff(\bftheta)) d\mu_m(\bftheta),
\end{equation*}
where $ \lambda_1(\bfff(\bftheta)), ..., \lambda_r(\bfff(\bftheta)) $ denote the eigenvalue of $ \bfff(\bftheta) $.
We say that $ \bfff $ is the spectral symbol of the sequence $ \left\{ X_{\bfnn,r} \right\}_\bfnn $ and denote it as $ \left\{ X_{\bfnn,r} \right\}_\bfnn \sim_{\lambda} \bfff $.\\
Note that, in the special case where $ r=1 $, the previous formula reads as
\begin{equation*}
\lim_{\bfnn \to \infty}  \Sigma_\lambda(F,X_{\bfnn,r}) = \frac{1}{\mu_m(D)} \int_{D} F(f(\bftheta) d\mu_m(\bftheta),
\end{equation*}
\end{definition}

When we consider a sequence of Toeplitz matrices generated by a Hermitian-valued function $ \bfff $ in $ L^1([-\pi,\pi]^d) $, it holds that $ \left\{ T_{\bfnn,r}(\bfff) \right\}_\bfnn \sim_{\lambda} \bfff $, that is, the generating function and the spectral symbol coincide (see \cite{Tillinota}).
The same is true regarding the preconditioned sequences and remarkably there are no outliers, thanks to the linear and positive nature of the underlying Toeplitz operators.

\begin{theorem}\label{th:toeplitz-summa}
$ \bfff \in L^1([-\pi,\pi]^d) $ be Hermitian-valued and let $ \bfg \in L^1([-\pi,\pi]^d) $ be Hermitian nonnegative definite valued with minimal eigenvalue not identically zero. Then $T_{\bfnn,r}(\bfff)$ is Hermitian and $T_{\bfnn,r}(\bfg)$ is positive definite for every dimension.
Furthermore
\begin{itemize}
\item $ \left\{ T_{\bfnn,r}(\bfff) \right\}_\bfnn \sim_{\lambda} \bfff $;
\item $ \left\{ T_{\bfnn,r}^{-1}(\bfg)T_{\bfnn,r}(\bfff) \right\}_\bfnn \sim_{\lambda} \bfg^{-1}\bfff $;
\item all the eigenvalues of $T_{\bfnn,r}^{-1}(\bfg)T_{\bfnn,r}(\bfff)$ belong to the open interval $(m,M)$ if $m= {\rm essinf}\, \lambda_{\min}(\bfg^{-1}\bfff )$,
$M={\rm esssup}\, \lambda_{\max}(\bfg^{-1}\bfff )$, and the minimal and maximal eigenvalue functions of $\bfg^{-1}\bfff$ are nonconstant almost everywhere (a.e.);
\item in the case where the minimal eigenvalue functions of $\bfg^{-1}\bfff$ is constant a.e. the smallest eigenvalue of $T_{\bfnn,r}^{-1}(\bfg)T_{\bfnn,r}(\bfff)$ may be equal to $m$ (analogously, in the case where the maximal eigenvalue functions of $\bfg^{-1}\bfff$ is constant a.e. the largest eigenvalue of $T_{\bfnn,r}^{-1}(\bfg)T_{\bfnn,r}(\bfff)$ may be equal to $M$).
\end{itemize}
\end{theorem}

Finally, we say that a sequence $ \left\{ X_{\bfnn,r} \right\}_\bfnn $ is zero distributed, denoted \mbox{$\left\{ X_{\bfnn,r} \right\}_\bfnn \sim_{\sigma} 0 $} if, for every $ F \in C_c(\mathbb{R}) $,
\begin{equation*}
\Sigma_\sigma(F,X_{\bfnn,r})=F(0).
\end{equation*}

\section{Spectral Analysis: localization results}\label{sec:sp_loca}
Consider now the preconditioned matrix $\cP^{-1}\cA$, where $\cA$ is  defined in \rref{trans_kronfrom} and $\cP$ is defined in \rref{eq_precPL}.
By standard algebraic manipulations the preconditioned matrix takes the form
\begin{equation}\label{eq_remterm}
\begin{array}{rcl}
\cP_L^{-1}\cA&=& (L_q\otimes M + \tau I_q\otimes K)^{-1}(A_q^{-1}\otimes M + \tau I_q\otimes K) \\
&=&(L_q\otimes M + \tau I_q\otimes K)^{-1}\left[(L_q\otimes M + \tau I_q\otimes K)+L_q\widehat{U}_q\otimes M\right] \\
&=& I_{qn} + (L_q\otimes M + \tau I_q\otimes K)^{-1}(L_q\widehat{U}_q\otimes M)\\
&=&I_{qn} + \underbrace{(I_{qn} + \tau (L_q^{-1}\otimes M^{-1}K))^{-1}}_{W_1^{-1}}\underbrace{(\widehat{U}_q\otimes I_n)}_{W_2},
\end{array}
\end{equation}
and, hence, the analysis reduces to the study of the spectrum of $W_1^{-1}W_2$. In \cite{IRKtheory} the spectral localization has been studied by using the field of values. Here we consider a more direct approach by setting explicitly the eigenvalue-eigenvector problem, by exploiting the lower-triangular and the strictly upper-triangular structure of the factors $L_q$ and $U_q$, correspondingly, which allow explicit computations.

Indeed, we consider the eigenvalues of $W_1^{-1}W_2$, that is, we set the basic relationships
\[
W_1^{-1}W_2 \bm{v} = \lambda \bm{v} \Leftrightarrow W_2 \bm{v} = \lambda W_1 \bm{v},
\]
where the second equation can be written as
\begin{equation}\label{transfromedeig}
(\widehat{U}_q\otimes I_n) \bm{v}   = \lambda (I_{qn} + \tau (L_q^{-1}\otimes M^{-1}K))  \bm{v}  =: \lambda \left(I_{qn} +  L_q^{-1}\otimes Z_\tau\right)  \bm{v},
\end{equation}
with $Z_\tau=\tau M^{-1}K$. The idea behind this reformulation is that the eigenvalues $\lambda$ of $W_1^{-1}W_2$ can be expressed as the eigenvalues of an explicit function of rational nature in terms of the matrix $Z_\tau$. We notice that the spectral behavior of $Z_\tau$ is well understood using the tools in the previous section (see Theorem \ref{th:toeplitz-summa}) and hence our problem of identifying a precise localization of $\lambda$ is substantially simplified.

We present a complete analysis of the cases $q=2$ in Section \ref{q=2} and $q=3$ in Section \ref{q=3}. The general setting is discussed in Section \ref{q>3}.
We stress that our findings are more precise than those in \cite{IRKtheory}, where the field of values is used as main tool. Indeed, our direct approach allows to obtain substantial generalizations and tighter localization results.

As anticipated in the introduction, we will assume that the matrices $K$ and $M$ will be both symmetric and positive definite: more precisely, following the notations in Section \ref{ssec:toeplitz-distribution}, we consider  $M=T_{\bm{n},1}(g_1)$, $h^2 K=T_{\bm{n},1}(g_2)$, with $g_1$ being a strictly positive trigonometric polynomial in the variable $\theta_1,\theta_2$, $g_2(\theta_1,\theta_2)=4-2\cos(\theta_2)-2\cos(\theta_2)$, $\bm{n}=(n,n)$.
However, our analysis can be generalized and we will discuss this issue in the conclusions.

\subsection{The two stage case}\label{q=2}

For the case $q=2$, we have
$$
A_2^{-1}=\begin{bmatrix}\frac{3}{2}&0\\-\frac{9}{2}&4\end{bmatrix}\begin{bmatrix}1&\frac{1}{3}\\0&1\end{bmatrix}, \quad\text{thus,} \quad
\begin{array}{rcl}
\widehat{U}_2 & = &
\begin{bmatrix}
0&\frac{1}{3}\\0&0
\end{bmatrix}, \quad
L_2^{-1}  =
\begin{bmatrix}
\frac{2}{3} & 0 \\ \frac{3}{4} & \frac{1}{4}
\end{bmatrix}.
\end{array}
$$
Let $\bm{v} = [\bm{v}_1 , \bm{v}_2 ]^T$ with $ \bm{v}_i \in \mathbb{C}^n$, $i=1,2$. As a consequence, taking into consideration \eqref{transfromedeig}, we obtain
\begin{equation}\label{eq_q2_0}
\begin{bmatrix}0&\frac{1}{3} I_n\\0&0\end{bmatrix} \begin{bmatrix}
    \bm{v}_1 \\ \bm{v}_2 \end{bmatrix} = \lambda \Bigg{(} \begin{bmatrix}
    \bm{v}_1 \\ \bm{v}_2 \end{bmatrix} +  \begin{bmatrix}
    \frac{2}{3}Z_\tau & 0 \\ \frac{3}{4}Z_\tau & \frac{1}{4}Z_\tau
\end{bmatrix} \begin{bmatrix}
    \bm{v}_1 \\ \bm{v}_2 \end{bmatrix} \Bigg{)}.
\end{equation}
We first note that \rref{eq_q2_0} is satisfied for $\lambda=0$ and $\bm{v}_2 =0$ for all $\bm{v}_1 \in \mathbb{C}^n$, i.e., the eigenvalue $\lambda = 0$ has geometric multiplicity at least $n$. For localizing the remaining eigenvalues, we assume $\lambda \neq 0$. In that case, by dividing by $\lambda$, the second block row is transformed as
\[
\bm{v}_2 + \frac{3}{4} Z_\tau \bm{v}_1 + \frac{1}{4} Z_\tau \bm{v}_2 = 0 \Leftrightarrow \left(\frac{1}{3} I_n  + \frac{4}{3} Z_\tau^{-1}\right)\bm{v}_2 = - \bm{v}_1.
\]
In this way $\bm{v}_1$ is expressed as a function of $Z_\tau$ and $\bm{v}_2$. Now the first block row becomes
\[
\frac{1}{3} \bm{v}_2 = \lambda \left( I_n + \frac{2}{3} Z_\tau \right)\bm{v}_1
\]
and, as a consequence, by inserting the explicit form of $\bm{v}_1$ from the first block row, we deduce
\[
\frac{1}{3} \bm{v}_2 = -\lambda \left( I_n + \frac{2}{3} Z_\tau\right) \left(\frac{1}{3} I_n  + \frac{4}{3} Z_\tau^{-1}\right)\bm{v}_2,
\]
which can be rewritten as
\[
\bm{v}_2 = -\lambda ( 3I_n + 2Z_\tau ) \left(\frac{1}{3} I_n  + \frac{4}{3} Z_\tau^{-1}\right)\bm{v}_2 = - \lambda \left(4 Z_\tau^{-1} + \frac{2}{3} Z_\tau + \frac{11}{3} I_n \right)\bm{v}_2.
\]
As a final step we find
\[
 -\left(4 Z_\tau^{-1} + \frac{2}{3} Z_\tau + \frac{11}{3} I_n\right)^{-1} \bm{v}_2  = \lambda \bm{v}_2.
\]
The expression above is crucial since the same nonzero eigenvalues $\lambda$ are exactly those of the rational matrix function
\[
f(Z_\tau)=-\left(4 Z_\tau^{-1} + \frac{2}{3} Z_\tau + \frac{11}{3} I_n\right)^{-1}.
\]
Therefore, if $\mu_\tau$ is the generic eigenvalue of $Z_\tau$ then the generic nonzero eigenvalue $\lambda$ of our original problem is
\begin{equation}\label{f-q=2}
f(\mu_\tau)=-\left(4 \mu_\tau^{-1} + \frac{2}{3} \mu_\tau + \frac{11}{3} \right)^{-1}.
\end{equation}
It is now insightful to notice that, independently of the mesh parameters $h$, $\tau$, $n$, since $K$ and $M$ are both positive definite, we infer that $\mu_\tau \in (0,\infty)$ and, hence, $f(\mu_\tau)<0$. We first notice that
\[
\lim_{\mu_\tau\rightarrow 0^+} f(\mu_\tau) = \lim_{\mu_\tau\rightarrow +\infty} f(\mu_\tau)=0,
\]
while, setting
\[
g(\mu_\tau)=-\frac{1}{f(\mu_\tau)}=4 \mu_\tau^{-1} + \frac{2}{3} \mu_\tau + \frac{11}{3},
\]
we find that $g'$ has a unique zero at $\mu_\tau^*=\sqrt{6}$ so that
\[
\min_{\mu_\tau> 0} f(\mu_\tau)= f(\sqrt{6})=-\frac{3\sqrt{6}}{11\sqrt{6}+24}=-r^*_2.
\]
In this way, since $Z_\tau$ is diagonalizable, it is proven that all the eigenvalues of the preconditioned matrix $\cP_L^{-1}\cA$ are either $1$ with algebraic and geometric multiplicity exactly equal to $n$ or they belong to the quite small real positive interval $[1-r^*_2,1)$, where this claim perfectly agrees with the numerical results and improves substantially the previous analysis in \cite{IRKtheory}.
The value of $r^*_2$ is approximately $0.144$, thus, the eigenvalues are located in the interval $[0.8558,1]$.
In addition it should be observed that this localization interval cannot be improved if we do not give constraints on the approximation parameters $h$, $\tau$, $n$ and this is also confirmed in Figure \ref{fig_Ivo1-1}.

\begin{figure}[htbp]
\centering
\includegraphics[width=0.32\textwidth]{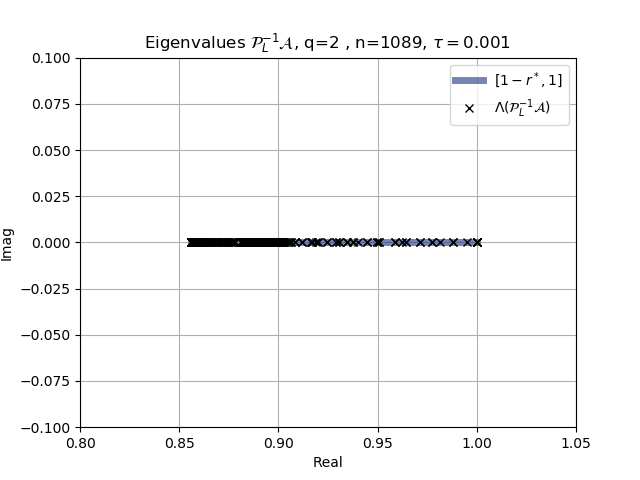} 
\includegraphics[width=0.32\textwidth]{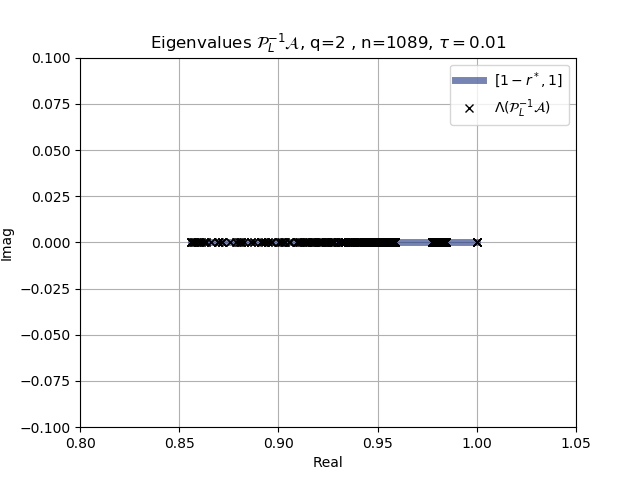} 
\includegraphics[width=0.32\textwidth]{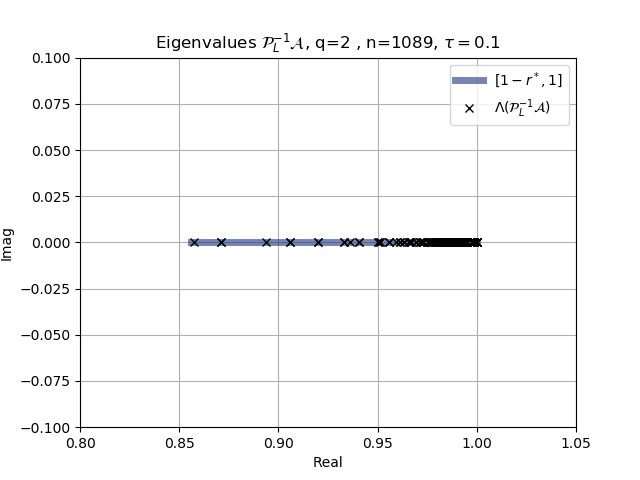} 
\caption{$q=2$: Eigenvalues of
$\cP_L^{-1}\cA$ - for a range of $\tau \in \{0.001,0.01,0.1$\} }
\label{fig_Ivo1-1}
\end{figure}

All the above derivations can be put in a unique result, which is substantially stronger than Theorem 1 in \cite{IRKtheory}.

\begin{theorem}\label{eigentructure-q=2}
Let $\cP_L^{-1}\cA$ be the preconditioned matrix defined in (\ref{eq_remterm}), with $K,M$ being symmetric positive definite stiffness and mass matrices, respectively, and $q=2$. Then the following properties of the eigenvalues $\lambda$ and the eigenvectors $\bm{v}$ of $\cP_L^{-1}\cA$ hold:
\begin{itemize}
\item $\lambda=1$ with algebraic and geometric multiplicity equal to $n$;
\item $1+f(\mu_\tau)\in [1-r^*_2,1)$, where $\mu_\tau$ is any eigenvalue of $Z_\tau=\tau M^{-1}K$, $f$ is defined in (\ref{f-q=2}), and range$(f)=[-r^*_2,0)$ with
\[
-r^*_2=-\frac{3\sqrt{6}}{11\sqrt{6}+24}=\min_{\mu\in (0,\infty)} f(\mu),
\]
with $\lim_{\mu_\tau=0^+,+\infty}f(\mu_\tau)=0^-$;
\item The eigenvectors related to $\lambda=1$ take the form
\[
\bm{v}=
   \begin{bmatrix}
    \bm{v}_1 \\ \bm{0}
   \end{bmatrix}
\]
for all vectors $\bm{v}_1$ of size $n$. Furthermore, the eigenvector associated with the eigenvalue $f(\mu_\tau)$ has the specific expression
\[
\bm{v}=
   \begin{bmatrix}
    \frac{1}{3}\left(1+\mu_\tau^{-1}\right)\bm{v}_2 \\ \bm{v}_2
   \end{bmatrix}
\]
with $\bm{v}_2$ nonzero vector such that $Z_\tau \bm{v}_2= \mu_\tau\bm{v}_2$.
\end{itemize}
\end{theorem}

\begin{remark}[\rm regarding the range of $f$ vs the behavior of the spectrum of $Z_\tau$]\label{rem:rang f vs eigs of Ztau}
In accordance with the notations in Section \ref{ssec:toeplitz-distribution}, first we recall that $M=T_{\bm{n},1}(g_1)$, $h^2 K=T_{\bm{n},1}(g_2)$, $g_1$ is strictly positive trigonometric polynomial in the variable $\theta_1,\theta_2$, $g_2(\theta_1,\theta_2)=4-2\cos(\theta_2)-2\cos(\theta_2)$, $\bm{n}=(n,n)$. Hence, by the standard spectral theory of multilevel Toeplitz matrices and matrix-sequences recalled in the third item of Theorem \ref{th:toeplitz-summa} with $r=1$ (see \cite{GLT1_book,GLT2_book} and references therein), we know that the eigenvalues of $h^2M^{-1}K=T_{\bm{n}}^{-1}(g_1)T_{\bm{n}}(g_2)$ belong to the open interval
\[
\left(\min \frac{g_2}{g_1}, \max \frac{g_2}{g_1}\right)=\left(0,\max \frac{g_2}{g_1}\right).
\]
Therefore, the eigenvalues of
\[
Z_\tau = \tau M^{-1}K = \frac{\tau}{h^2}T_{\bm{n}}^{-1}(g_1)T_{\bm{n}}(g_2)
\]
belong to the interval
\[
I_{h,\tau}=\left(0,\frac{\tau}{h^2}\max \frac{g_2}{g_1}\right).
\]
Now we recall that the used IRK method has precision in time $\tau^{2q-1}=\tau^{3}$ for $q=2$ and precision in space $h^2$. To balance the space and time discretization errors we assume $\tau^3\sim h^2$ so that
\[
I_{h,\tau}=\left(0,C \tau^{-2}\max \frac{g_2}{g_1}\right)
\]
for some positive constant $C$ independent of $h$ and $\tau$. As a consequence, the interval tends to $(0,+\infty)$ as $\tau$ tends to zero. In other words, the estimates, given in Theorem \ref{eigentructure-q=2}, are tight and cannot be improved, unless we impose artificial constraints on the parameters $\tau, h$.
\end{remark}

\subsection{The three stage case}\label{q=3}

In the case of $q=3$, the relevant matrices are the following
\begin{equation}\label{eq_q3_U3}
\widehat{U}_3 = \begin{bmatrix}
    0 & \displaystyle{\frac{87\sqrt{6} - 108}{45\sqrt{6} + 180}} & -\displaystyle{\frac{2(300\sqrt{6} - 450)}{25(45\sqrt{6} + 180)}} \\
    0 & 0 & \displaystyle{\frac{ \displaystyle{\frac{4\sqrt{6}}{15} - \frac{((87\sqrt{6} + 108)(300\sqrt{6} - 450))}{(1125(45\sqrt{6} + 180))} + \frac{2}{5}}}{   \displaystyle{\frac{(87\sqrt{6} - 108)(87\sqrt{6} + 108)}{90(45\sqrt{6} + 180)} - \frac{\sqrt{6}}{2} + 2}}} \\
    0 & 0 & 0
\end{bmatrix},
\end{equation}
\begin{equation}\label{eq_q3_L3}
L_3^{-1} = \begin{bmatrix}
       \displaystyle{\frac{90}{(45\sqrt{6} + 180)}} & 0 & 0 \\
      \displaystyle{\frac{(29 \sqrt{6})}{200} + \frac{9}{50}} &  \displaystyle{\frac{3\sqrt{6}}{40} + \frac{3}{10}} & 0 \\
       \displaystyle{\frac{4}{9} - \frac{\sqrt{6}}{36}} & \displaystyle{\frac{4}{9} + \frac{\sqrt{6}}{36}} & \displaystyle{\frac{1}{9}}
\end{bmatrix}.
\end{equation}
We follow exactly the same idea as in the case of $q=2$.
In an analogous manner it is seen that first $\lambda=0$ has a geometric multiplicity of at least $n$. Second, given $\mu_\tau$ a generic eigenvalue of $Z_\tau$, there are two eigenvalues $\lambda$ of the preconditioned matrix satisfying a nondegenerate second degree polynomial of the form  $a\lambda^2 + b\lambda + c$, where $a, b, c$ are real-valued rational functions of $\mu_\tau$. This is in perfect agreement with the two complex branches that appear in the numerical plots in Figure \ref{fig_Ivo1-2a}. 
   In addition, since $Z_\tau$ is diagonalizable (in fact it is similar to a positive definite matrix)  and has size $n$, the eigenvalues $\lambda$ determined by this procedure are $2n$ which is exactly what we expect and, hence, as for $q=2$, the eigenvalue $\lambda=0$ of $W_1^{-1}W_2$ has algebraic and geometric multiplicity exactly equal to $n$.

Now we proceed with the calculation and since the coefficients of the factors have rather complicated expression, we perform a general computation of parametric type by denoting
\[
\widehat{U}_3 = \begin{bmatrix}
    0 & {\hat u}_{1,2} & {\hat u}_{1,3} \\
    0 & 0 & {\hat u}_{2,3} \\
    0 & 0 & 0
\end{bmatrix},
\qquad
L_3^{-1} = \begin{bmatrix}
       {\ell}_{1,1} & 0 & 0 \\
      {\ell}_{2,1} &  {\ell}_{2,2} & 0 \\
      {\ell}_{3,1} & {\ell}_{3,2} & {\ell}_{3,3}
\end{bmatrix}.
\]
Let $\bm{v} = [\bm{v}_1 , \bm{v}_2 , \bm{v}_3 ]^T$ with $ \bm{v}_i \in \mathbb{C}^n$, $i=1,2,3$. Relation  \eqref{transfromedeig} can be written as
\[
\begin{bmatrix}
    0 & {\hat u}_{1,2} & {\hat u}_{1,3} \\
    0 & 0 & {\hat u}_{2,3} \\
    0 & 0 & 0
\end{bmatrix} \begin{bmatrix}
    \bm{v}_1 \\ \bm{v}_2 \\ \bm{v}_3 \end{bmatrix} = \lambda \Bigg{(} \begin{bmatrix}
    \bm{v}_1 \\ \bm{v}_2 \\ \bm{v}_3 \end{bmatrix} + \begin{bmatrix}
       {\ell}_{1,1} & 0 & 0 \\
      {\ell}_{2,1} &  {\ell}_{2,2} & 0 \\
      {\ell}_{3,1} & {\ell}_{3,2} & {\ell}_{3,3}
\end{bmatrix}
 \begin{bmatrix}   Z_\tau \bm{v}_1 \\ Z_\tau \bm{v}_2 \\ Z_\tau \bm{v}_3 \end{bmatrix} \Bigg{)}.
\]
From the last block equation we have $\bm{0} = \lambda \left(\bm{v}_3 + \sum_{j=1}^3 {\ell}_{3,j} Z_\tau \bm{v}_j\right)$ which is satisfied for $\lambda=0$. By choosing $\bm{v}_1= \bm{v}_2 = \bm{0}$ and $\lambda=0$, we deduce that  \eqref{transfromedeig} holds for every choice of $\bm{v}_3$ so that $\lambda=0$ is an eigenvalue with geometric multiplicity at least $n$.
Let now $\lambda \neq 0$. The third block row leads to
\[
-( I_n + {\ell}_{3,3} Z_\tau) \bm{v}_3 = {\ell}_{3,1} Z_\tau  \bm{v}_1 + {\ell}_{3,2} Z_\tau  \bm{v}_2.
\]
Assuming that $I_n + {\ell}_{3,3} Z_\tau$ is invertible, we find that
\[
\bm{v}_3 = -( I_n + {\ell}_{3,3} Z_\tau)^{-1} \Big{(} {\ell}_{3,1} Z_\tau  \bm{v}_1 + {\ell}_{3,2} Z_\tau  \bm{v}_2 \Big{)}.
\]
Now we write the formal expression of the first and of the second block rows, that is,
\begin{eqnarray*}
{\hat u}_{2,3} \bm{v}_3 & = &\lambda \Big{(} \bm{v}_2 + {\ell}_{2,1} Z_\tau \bm{v}_1 +  {\ell}_{2,2} Z_\tau)Z \bm{v_2}     \Big{)}, \\
{\hat u}_{1,2} \bm{v}_2 + {\hat u}_{1,3}\bm{v}_3 & = &\lambda \Big{(} \bm{v}_1 + {\ell}_{1,1} Z_\tau \bm{v}_1      \Big{)}.
\end{eqnarray*}
Next we replace the explicit form of $\bm{v}_3$ in both equalities and obtain
\begin{eqnarray*}
-{\hat u}_{2,3} ( I_n + {\ell}_{3,3} Z_\tau)^{-1} \Big{(} {\ell}_{3,1} Z_\tau  \bm{v}_1 + {\ell}_{3,2} Z_\tau  \bm{v}_2 \Big{)} & = &\lambda \Big{(} \bm{v}_2 + {\ell}_{2,1} Z_\tau \bm{v}_1 +  {\ell}_{2,2} Z_\tau) \bm{v_2}     \Big{)}, \\
{\hat u}_{1,2} \bm{v}_2 - {\hat u}_{1,3}( I_n + {\ell}_{3,3} Z_\tau)^{-1} \Big{(} {\ell}_{3,1} Z_\tau  \bm{v}_1 + {\ell}_{3,2} Z_\tau  \bm{v}_2 \Big{)} & = &\lambda \Big{(} \bm{v}_1 + {\ell}_{1,1} Z_\tau \bm{v}_1      \Big{)}.
\end{eqnarray*}
The second block equation allows us to express  $\bm{v}_2$ as a function of $\lambda, Z_\tau$, and $\bm{v}_1$.
Indeed, setting
\begin{eqnarray*}
l_1(\lambda,Z_\tau) & = & -\left[{\hat u}_{2,3} ( I_n + {\ell}_{3,3} Z_\tau)^{-1}{\ell}_{3,2}Z_\tau + \lambda (I_n+{\ell}_{2,2} Z_\tau)\right],\\ l_2(\lambda,Z_\tau) & = & {\hat u}_{2,3} ( I_n + {\ell}_{3,3} Z_\tau)^{-1}{\ell}_{3,1} Z_\tau +\lambda {\ell}_{2,1} Z_\tau,
\end{eqnarray*}
we obtain an expression for $\bm{v}_2$,
$
\bm{v}_2 = l_1^{-1}(\lambda,Z_\tau)l_2(\lambda,Z_\tau)\bm{v}_1.
$

Finally we are ready for the last substitution in order to obtain a generalized eigenvalue problem involving only the vector $\bm{v}_1$.
Taking into account that, since $Z_\tau$ is similar to a positive definite matrix, it  is diagonalizable,  choosing $\bm{v}_1$ as the eigenvector  of $Z_\tau$ associated with the eigenvalue $\mu_\tau$, we obtain the relation
\begin{equation}\label{lambda second degree}
-\lambda (1 + {\ell}_{1,1} \mu_\tau)(1 + {\ell}_{3,3} \mu_\tau)l_1(\lambda,\mu_\tau)
- {\hat u}_{1,3}{\ell}_{3,2}\mu_\tau l_2(\lambda,\mu_\tau) + {\hat u}_{1,2}(1 + {\ell}_{3,3} \mu_\tau)l_2(\lambda,\mu_\tau)=0.
\end{equation}
Since $l_1(\lambda,\mu_\tau)$ and $l_2(\lambda,\mu_\tau)$ are first degree polynomials in the variable $\lambda$, the global resulting equation
(\ref{lambda second degree}) is a nondegenerate second degree equation in the variable $\lambda$.

As a consequence, the solution is given by two branches $\lambda_1(\mu_\tau), \lambda_2(\mu_\tau)$ which are irrational, nontrascendental functions of $\mu_\tau\in (0,\infty)$: recall that the eigenvalues of the whole preconditioned matrix are $1$ with algebraic and geometric multiplicity $n$,
$1+\lambda_1(\mu_\tau), 1+\lambda_2(\mu_\tau)$ for every $\mu_\tau$ eigenvalue of the diagonalizable matrix $Z_\tau$.

The analysis in the specific setting of $q=3$ and with the specific coefficients from \rref{eq_q3_U3} and \rref{eq_q3_L3} allows to claim the following:
\begin{itemize}
\item the terms $I_n+\ell_{i,i}Z_\tau$, $i=1,2,3$, are all invertible since $Z_\tau$ is similar to a positive definite matrix and because $\ell_{i,i}$, $i=1,2,3$, are all positive coefficients;
\item by substituting the values of the parameters, the remaining eigenvalues of the preconditioned matrix lie in a disk centered in $1$ and of radius $r^*_3<1$;
\item furthermore, by substituting the values of the parameters, the asymptotic analysis of equation (\ref{lambda second degree}) implies that for $\mu_\tau\rightarrow 0^+$, the two eigenvalues $\lambda_1(f(\mu_\tau)), \lambda_2(f(\mu_\tau))$ are complex conjugate tending to $0$ and with their real part converging to $0^-$ as $\mu_\tau\rightarrow 0^+$, with Re$(\lambda_1(f(\mu_\tau)))=$Re$(\lambda_2(f(\mu_\tau))=o($Im$(\lambda_1(f(\mu_\tau))))$ and
Im$(\lambda_1(f(\mu_\tau)))=-$Im$(\lambda_2(f(\mu_\tau)))$;
\item finally, by substituting the values of the parameters, the asymptotic analysis of equation (\ref{lambda second degree}) implies that for $\mu_\tau\rightarrow +\infty$, the two eigenvalues $\lambda_1(f(\mu_\tau))$, $\lambda_2(f(\mu_\tau))$ are complex conjugate tending to $0$ and with their real part converging to $0^-$ as $\mu_\tau\rightarrow +\infty$, with $$\text{Re}(\lambda_1(f(\mu_\tau)))=\text{Re}(\lambda_2(f(\mu_\tau))=o(\text{Im}(\lambda_1(f(\mu_\tau))))$$ and Im$(\lambda_1(f(\mu_\tau)))=-$Im$(\lambda_2(f(\mu_\tau)))$.
\end{itemize}
 The above results are resumed in Theorem \ref{eigentructure-q=3}.

\begin{theorem}\label{eigentructure-q=3}
Let $\cP_L^{-1}\cA$ be the preconditioned matrix defined in (\ref{eq_remterm}), with $K,M$ being symmetric positive definite stiffness and mass matrices, respectively, and $q=3$. Then the eigenvalues of $\cP_L^{-1}\cA$ are
\begin{itemize}
\item $\lambda=1$ with algebraic and geometric multiplicity equal to $n$;
\item $1+\lambda_1(f(\mu_\tau))$ and $1+\lambda_2(f(\mu_\tau))$ where $\mu_\tau$ is any eigenvalue of $Z_\tau=\tau M^{-1}K$, $f$ is defined in (\ref{lambda second degree});
\item All the eigenvalues of the preconditioned matrix lie in a disk centered in $1$ and of radius $r^*_3<1$, {numerical calculation yields $r^*_3 \approx 0.206$};
\item For $\mu_\tau\rightarrow 0^+$, the two eigenvalues $\lambda_1(f(\mu_\tau))$ and $\lambda_2f((\mu_\tau))$ are complex conjugate tending to $0$ and with their real part converging to $0^-$ as $\mu_\tau\rightarrow 0^+$, with
\[
\begin{array}{l}
{\rm Re}(\lambda_1(f(\mu_\tau)))={\rm Re}(\lambda_2(f(\mu_\tau))=o({\rm Im}(\lambda_1(f(\mu_\tau)))), \\ \ {\rm Im}(\lambda_1(f(\mu_\tau)))=
-{\rm Im}(\lambda_2(f(\mu_\tau)));
\end{array}
\]
\item For $\mu_\tau\rightarrow +\infty$, the two eigenvalues $\lambda_1(f(\mu_\tau))$ and $\lambda_2(f(\mu_\tau))$ are complex conjugate tending to $0$ and with their real part converging to $0^-$ as $\mu_\tau\rightarrow +\infty$, with
\begin{align*}
    {\rm Re}(\lambda_1(f(\mu_\tau))) &={\rm Re}(\lambda_2(f(\mu_\tau))=o({\rm Im}(\lambda_1(f(\mu_\tau)))), \\
    {\rm Im}(\lambda_1(f(\mu_\tau))) &=- \rm Im (\lambda_2(f(\mu_\tau)));
\end{align*}

\item The eigenvectors related to $\lambda=1$ take the form
\[
\bm{v}=
   \begin{bmatrix}
    \bm{v}_1 \\ \bm{0} \\ \bm{0}
   \end{bmatrix}
\]
for all vectors $\bm{v}_1$ of size $n$. Furthermore, the eigenvectors associated to the eigenvalues $\lambda_1(f(\mu_\tau))$ and $\lambda_2(f(\mu_\tau))$ have the specific expression
\[
\bm{v}=
   \begin{bmatrix}
    \bm{v}_1 \\
    \alpha(\mu_\tau)\bm{v}_1 \\ \beta(\mu_\tau)\bm{v}_1
   \end{bmatrix}
\]
with $\bm{v}_1$ nonzero vector such that $Z_\tau \bm{v}_1= \mu_\tau\bm{v}_1$ and
\begin{align*}
\alpha(\mu_\tau) &=l_1^{-1}(\lambda,\mu_\tau)l_2(\lambda,\mu_\tau) \\
    \beta(\mu_\tau) &=-\mu_\tau( 1 + {\ell}_{3,3} \mu_\tau)^{-1}\left[{\ell}_{3,1}  + {\ell}_{3,2} \alpha(\mu_\tau)\right]
\end{align*}

as in the previous computations.
\end{itemize}
\end{theorem}

\begin{remark}[\rm regarding the range of $f$ vs the behavior of the spectrum of $Z_\tau$]\label{rem:rang f vs eigs of Ztau q=3}
As in Remark \ref{rem:rang f vs eigs of Ztau} and the following notation and results in Section \ref{ssec:toeplitz-distribution}, we have $M=T_{\bm{n},1}(g_1)$, $h^2 K=T_{\bm{n},1}(g_2)$, $g_1$ strictly positive trigonometric polynomial in the variable $\theta_1,\theta_2$, $g_2(\theta_1,\theta_2)=4-2\cos(\theta_2)-2\cos(\theta_2)$, $\bm{n}=(n,n)$. Following verbatim the same reasoning as  for $q=2$, the eigenvalues of $h^2M^{-1}K=T_{\bm{n},1}^{-1}(g_1)T_{\bm{n},1}(g_2)$ belong to the open interval
\[
\left(\min \frac{g_2}{g_1}, \max \frac{g_2}{g_1}\right)=\left(0,\max \frac{g_2}{g_1}\right)
\]
so that the eigenvalues of $Z_\tau = \tau M^{-1}K = \frac{\tau}{h^2}T_{\bm{n},1}^{-1}(g_1)T_{\bm{n},1}(g_2)$ all remain in the interval
\[
I_{h,\tau}=\left(0,\frac{\tau}{h^2}\max \frac{g_2}{g_1}\right).
\]
We use again the fact that the IRK method has accuracy in time $\tau^{2q-1}=\tau^{5}$ for $q=3$ and accuracy in space $h^2$. Hence, we assume $\tau^5\sim h^2$ so that
\[
I_{h,\tau}=\left(0,C \tau^{-4}\max \frac{g_2}{g_1}\right)
\]
for some positive constant $C$ independent of $h$ and $\tau$. As a consequence the interval tends to $(0,+\infty)$ as $\tau$ tends to zero. In other words, the estimates given in Theorem \ref{eigentructure-q=3} are tight and cannot be improved, unless we impose artificial constraints on the parameters $\tau, h$.
\end{remark}

\begin{figure}[htbp]
\centering
\includegraphics[width=0.4\textwidth]{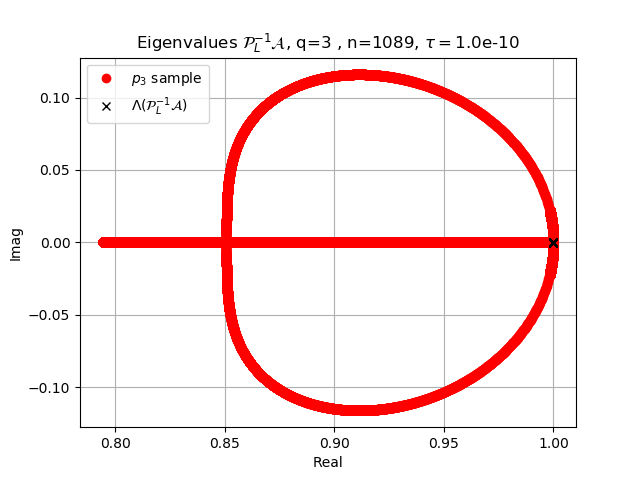} 
\includegraphics[width=0.4\textwidth]{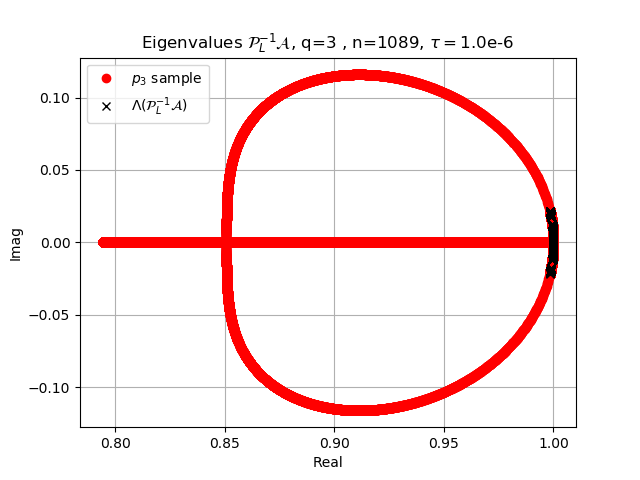} \\ 
\includegraphics[width=0.4\textwidth]{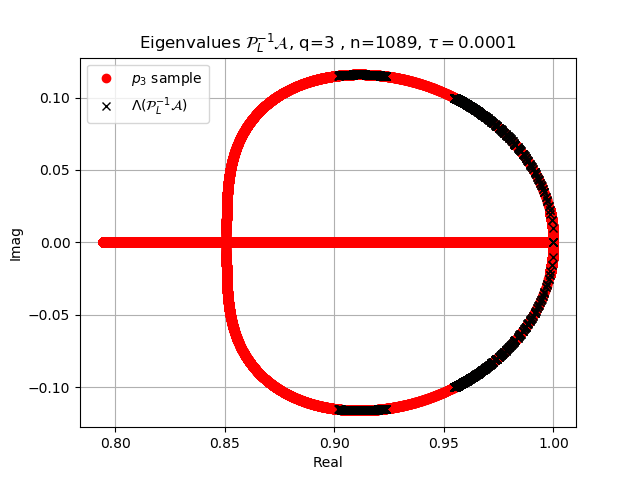} 
\includegraphics[width=0.4\textwidth]{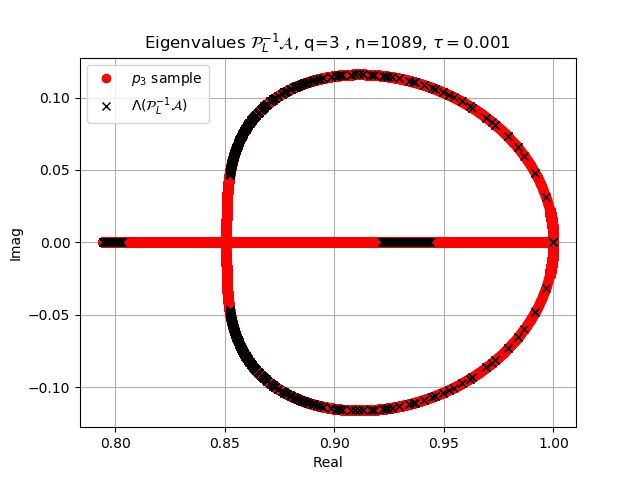} \\ 
\includegraphics[width=0.4\textwidth]{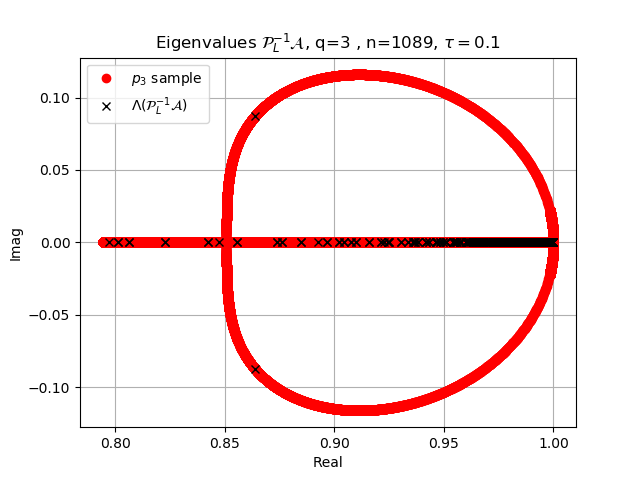}
\includegraphics[width=0.4\textwidth]{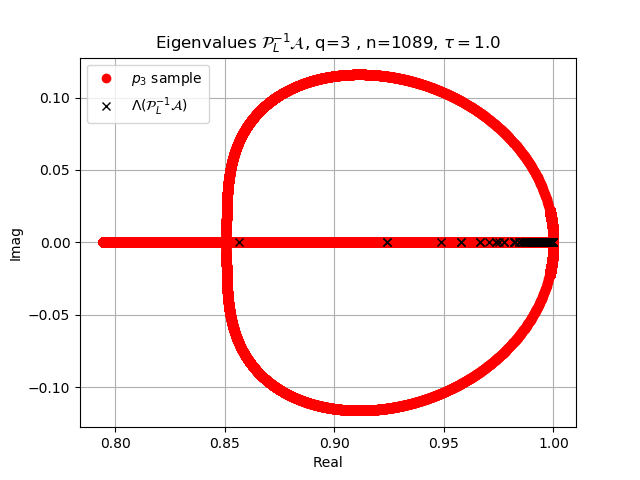} \\ 
\includegraphics[width=0.4\textwidth]{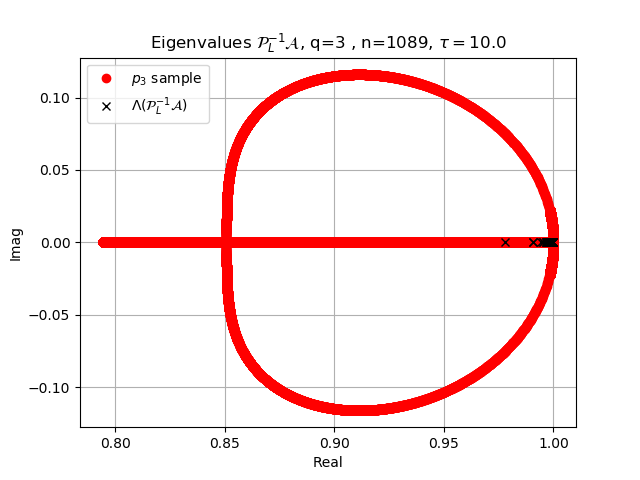}
\includegraphics[width=0.4\textwidth]{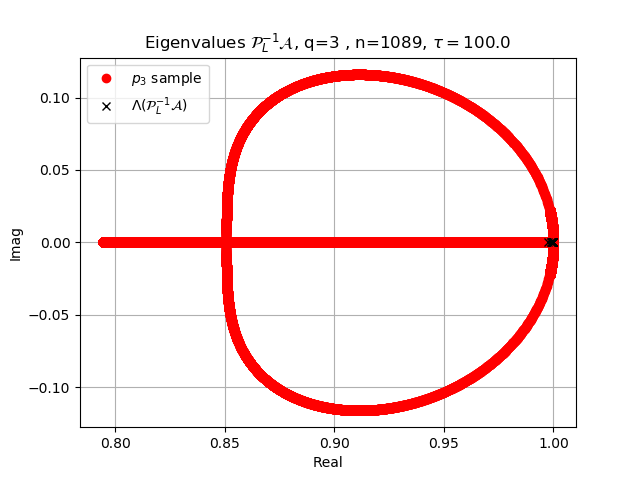}
\caption{$q=3$: Eigenvalues of $\cP_L^{-1}\cA$ - for a range of $\tau \in [10^{-10},10^2]$, keeping $h$ constant }
\label{fig_Ivo1-2a}
\end{figure}

\begin{figure}[htbp]
\centering
\includegraphics[width=0.49\textwidth]{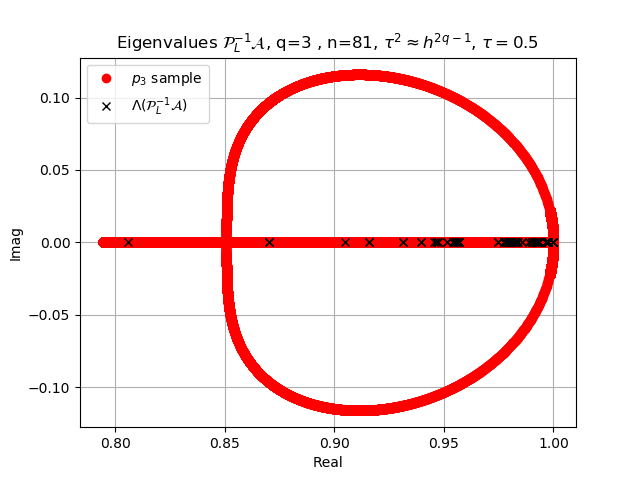}
\includegraphics[width=0.49\textwidth]{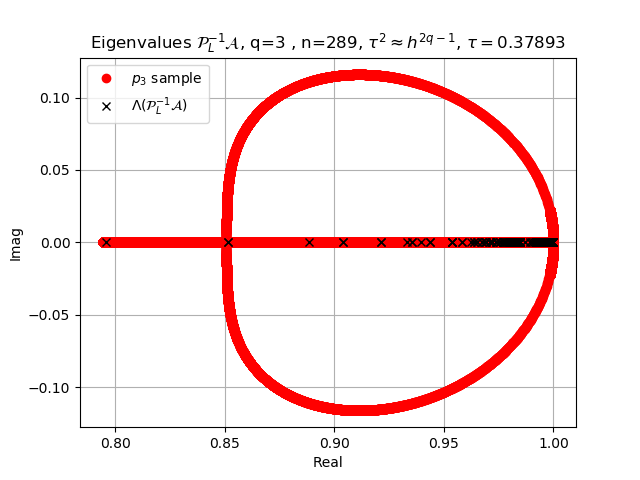}
\includegraphics[width=0.49\textwidth]{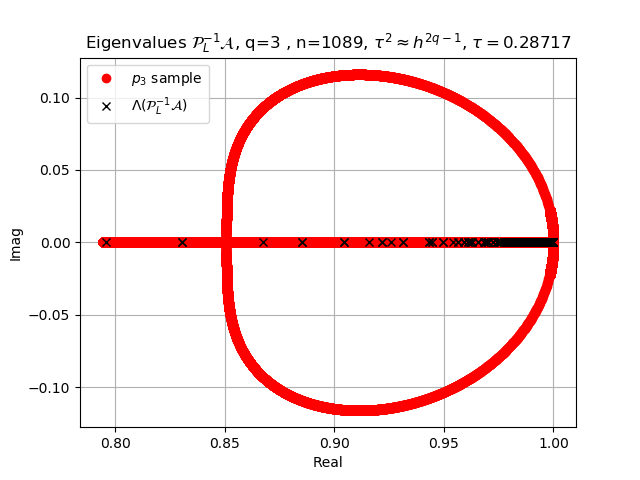}
\includegraphics[width=0.49\textwidth]{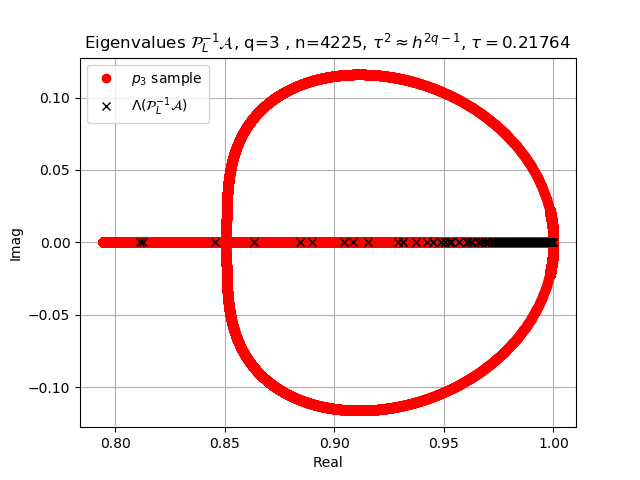}
\caption{$q=3$: Eigenvalues of $\cP_L^{-1}\cA$ - for a range of $h$ choosing $\tau$ such that $\tau^{2q-1} = h^2$.}
\label{fig_Ivo1-3}
\end{figure}

\begin{figure}[htbp]
\centering
\includegraphics[width=0.49\textwidth]{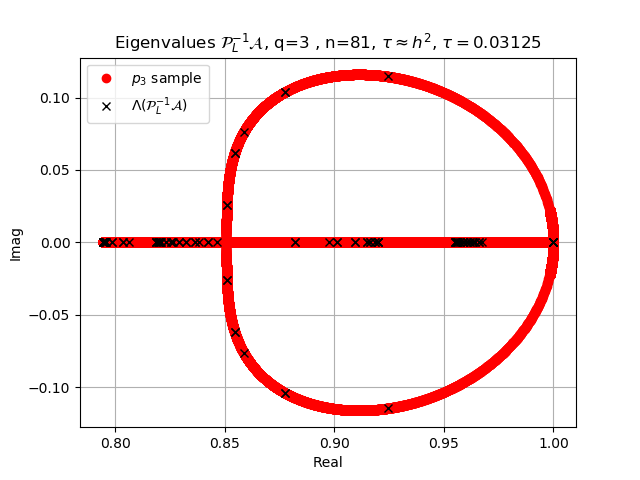}
\includegraphics[width=0.49\textwidth]{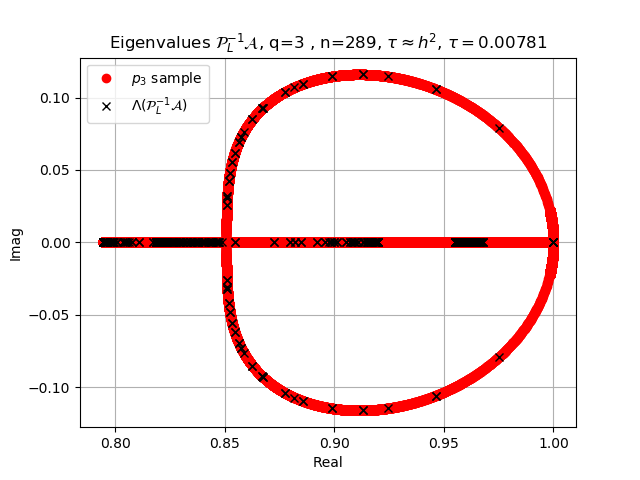}
\includegraphics[width=0.49\textwidth]{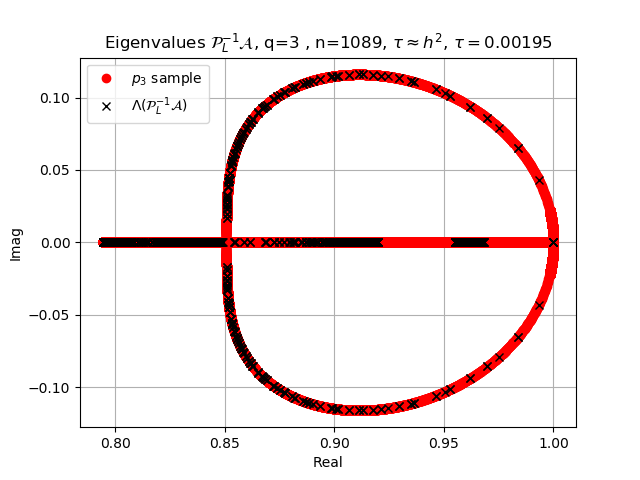}
\includegraphics[width=0.49\textwidth]{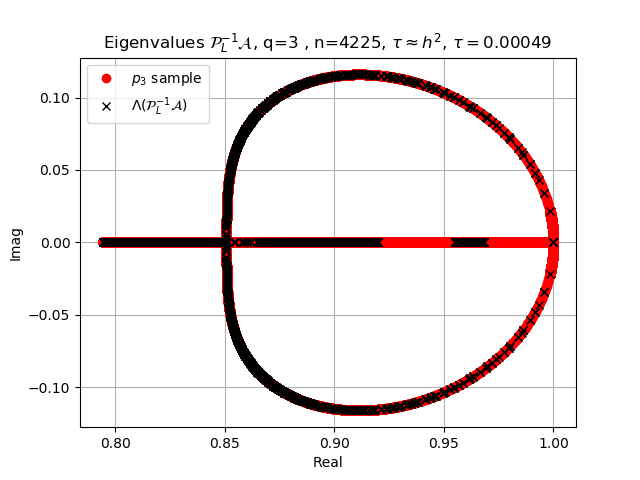}
\caption{$q=3$: Eigenvalues of $\cP_L^{-1}\cA$ - for a range of $h$ choosing $\tau$ such that $\tau = h^2$.}
\label{fig_Ivo1-4}
\end{figure}

\subsection{The general case with $q$ stages, $q>3$}\label{q>3}
 We emphasize that the procedure followed in Section \ref{q=3}  enables the analysis of the generic case $q>3$ since it is clear that the same scheme leads to a polynomial in $\lambda$ of degree $q-1$ and, hence, beside the $n$ trivial eigenvalues equal to $1$, to any of the $n$ eigenvalues $\mu_\tau$ of $Z_\tau$ there correspond exactly $q-1$ eigenvalues, which are the roots of a polynomial of degree $q-1$ with real-valued rational coefficients in $\mu_\tau$.

As a consequence, their solution is given by $q-1$ branches $\lambda_1(\mu_\tau), \ldots, \lambda_{q-1}(\mu_\tau)$ which are irrational, nontrascendental functions of $\mu_\tau\in (0,\infty)$: recall that the eigenvalues of the whole preconditioned matrix are $1$ with algebraic and geometric multiplicity $n$, $1+\lambda_1(\mu_\tau),\ldots, 1+\lambda_{q-1}(\mu_\tau)$ for every $\mu_\tau$ eigenvalue of the diagonalizable matrix $Z_\tau$.

The analysis in the specific setting of $q=4,5,6,7,8,9,10$ with the coefficients of the corresponding $\widehat{U}_q$ and $L_q^{-1}$ allows to claim the following:
\begin{itemize}
\item the procedure does not stop thanks to nonsingular character of $I_n+\ell_{i,i}Z_\tau$, $i=1,\ldots,q$, since $Z_\tau$ has positive eigenvalues (indeed it is similar to a positive definite matrix) and because $\ell_{i,i}$, $i=1,\ldots,q$, are all positive coefficients;
\item by substituting the values of the parameters, the remaining eigenvalues of the preconditioned matrix lie in a disk centered in $1$ and of radius $r^*_q<1$, $q=4,5,6,7,8,9,10$;
\item furthermore, by substituting the values of the parameters, the asymptotic analysis of the spectrum implies that for $\mu_\tau\rightarrow 0^+$, the eigenvalues $\lambda_1(f(\mu_\tau))$, $\ldots \lambda_{q-1}(f(\mu_\tau))$ tend to $0$ as $\mu_\tau\rightarrow 0^+$ and as $\mu_\tau\rightarrow +\infty$.
\end{itemize}

As in Remark \ref{rem:rang f vs eigs of Ztau} and Remark \ref{rem:rang f vs eigs of Ztau q=3}, $M=T_{\bm{n}}(g_1)$, $h^2 K=T_{\bm{n}}(g_2)$, $g_1$ is a strictly positive trigonometric polynomial in the variable $\theta_1,\theta_2$, $g_2(\theta_1,\theta_2)=4-2\cos(\theta_2)-2\cos(\theta_2)$, $\bm{n}=(n,n)$. Following verbatim the same reasoning as for $q=2$ and $q=3$, the eigenvalues of $h^2M^{-1}K=T_{\bm{n}}^{-1}(g_1)T_{\bm{n}}(g_2)$ belong to the open interval
\[
\left(\min \frac{g_2}{g_1}, \max \frac{g_2}{g_1}\right)=\left(0,\max \frac{g_2}{g_1}\right)
\]
so that the eigenvalues of $Z_\tau = \tau M^{-1}K = \frac{\tau}{h^2}T_{\bm{n}}^{-1}(g_1)T_{\bm{n}}(g_2)$ all remain in the interval
\[
I_{h,\tau}=\left(0,\frac{\tau}{h^2}\max \frac{g_2}{g_1}\right).
\]
As before, since the used IRK method has accuracy in time $\tau^{2q-1}$ and accuracy in space $h^2$, we assume $\tau^{2q-1}\sim h^2$ and therefore we infer
\[
I_{h,\tau}=\left(0,C \tau^{2-2q}\max \frac{g_2}{g_1}\right)=\left(0,C h^{-2 + \frac{2}{2q-1}}\max \frac{g_2}{g_1}\right)
\]
for some positive constant $C$ independent of $h$ and $\tau$ and with $\frac{2}{2q-1}$ belonging to $(0,\frac{2}{3}]$ and tending to zero as $q$ tends to infinity. As a consequence the interval tends to $(0,+\infty)$ as $\tau$ and $h$ tend to zero. In conclusion, the estimates given in Theorem \ref{eigentructure-q=3} are tight and cannot be improved, unless we impose artificial constraints on the parameters $\tau, h$.

\section{Spectral Analysis: distribution results}\label{sec:sp_distr}
In this section we collect the spectral results of global distributional type in the spirit of Definition \ref{def:spectral-distr}. Generally speaking these findings are difficult to obtain especially in a non Hermitian setting. However, in the current context, given the explicit expression found in Theorem \ref{eigentructure-q=2} for $q=2$, Theorem \ref{eigentructure-q=3} for $q=3$, and Section \ref{q>3} for values of $q$ larger than $3$, the distributional results become a straightforward consequence of Definition \ref{def:spectral-distr}, given the degree of freedom represented by the choice of the test functions.

\begin{theorem}\label{th:distr1}
Assuming that $\{Z_\tau\}_{\bfn}\sim_{\lambda} s$, ${\bfn}=(n,n)$, $Z_\tau=\tau M^{-1} K$, i.e., the matrix-sequence  $\{Z_\tau\}_{\bfn}$ is spectrally distributed as the measurable function $s$, then the preconditioned matrix-sequence  $\{\cP_L^{-1}\cA\}$ enjoys the relation
\[
\{\cP_L^{-1}\cA\}\sim_{\lambda} \bft_s
\]
with $\bft_s={\rm diag}(1,1+\lambda_1(f(s)),\ldots,1+\lambda_{q-1}(f(s))$. Here $\lambda_1(f)=f$ with $f$ as in (\ref{f-q=2}), if $q=2$, $f$ is as in (\ref{lambda second degree}) if $q=3$, while the general form of $f$ is deduced as in the procedure sketched in Section \ref{q>3} for $q>3$.
\end{theorem}
\begin{proof}
First we observe that the first branch given by the constant $1$ is produced by the $n$ eigenvalues of $\cP_L^{-1}\cA$ exactly equal to $1$. The other $q-1$ branches come from the fact that the other eigenvalues of $\cP_L^{-1}\cA$ are of the form
\[
1+ \lambda_i(f(\mu_\tau)), \ \ \ i=1,\ldots,q-1,
\]
for any eigenvalue $\mu_\tau$ of $Z_\tau=\tau M^{-1} K$. Since $\{\tau M^{-1} K\}_{\bfn}\sim_{\lambda} s$ the result follows directly from Definition \ref{def:spectral-distr}.
\end{proof}

\begin{theorem}\label{th:distr2}
If the discretization parameters $\tau$ and $h$ satisfy the order condition of optimal balancing, i.e., $\tau^{2q-1}\sim h^2$, then the spectral symbol $s$ of the matrix-sequence  $\{\tau M^{-1} K\}_{\bfn}$ is $s=+\infty$ so that $\bft_s = I_q$, that is the preconditioned matrix-sequence $\{\cP_L^{-1}\cA\}$ is spectrally clustered at $1$.
\end{theorem}
\begin{proof}
The statement is a plain consequence of Theorem \ref{th:distr1} taking into account that
\[
\lim_{\mu_\tau=0^+,+\infty} 1+ \lambda_i(f(\mu_\tau))=1, \ \ \ i=1,\ldots,q-1,
\]
thanks to Theorem \ref{eigentructure-q=2} for $q=2$, Theorem \ref{eigentructure-q=3} for $q=3$, and Section \ref{q>3} for larger values of $q$.
\end{proof}

\begin{remark}\label{other cases}
If $\tau=o(h^2)$ then $s=0$ identically and again $\bft_s = I_q$ that is the preconditioned matrix-sequence $\{\cP_L^{-1}\cA\}$ is spectrally clustered at $1$. Conversely, in the case where
\[
\frac{\tau}{h^2}=C,
\]
we observe the only case in which the spectral symbol $s$ is nontrivial i.e.
\[
s(\theta_1,\theta_2) = C \frac{g_2(\theta_1,\theta_2)}{g_1(\theta_1,\theta_2)}\in \left[0, C \max\frac{g_2}{g_1} \right].
\]
In such a setting, the more $C$ is large, the more we can appreciate the emergence of the $q$ spectral branches described in Theorem \ref{th:distr1}.
\end{remark}

\section{Numerical experiments}\label{sec:numerics}
In order to support the results regarding the distribution
 of the eigenvalues of the preconditioned system $\cP^{-1}\cA$, as derived is Theorem \ref{th:distr1}, Theorem \ref{th:distr2}, and Remark \ref{other cases}, we show numerically that for large enough $n$ the eigenvalues of  $\cP^{-1}\cA$ behave like the indicated distribution functions, as theoretically predicted. In the numerical tests $M$ and $K$ in \eqref{trans_kronfrom} and \eqref{eq_precPL} and are both generated with the deal.II FEM library using a Cartesian discretization using $Q_1$ bi-linear finite elements with $M$ being the mass matrix and $K$ being the discretized Laplace operator $-\Delta$ on the unit square $\Omega=[0,1]\times[0,1]$. The computations of the eigenvalue are done in Julia. Below, $h$ denotes the space discretization parameter, $n$ denotes the spacial dimension i.e., dimension of $K$ and $M$ and $n$ and $h$ are related as $h=1/(\sqrt{n}-1)$; dim($\cP_L^{-1}\cA$)$=qn$.
\subsection{Test 1}
With reference to Theorem \ref{th:distr2}, we define the following quantities
\[
N(\epsilon, h)= \# \{j:  |\lambda_j(\cP_L^{-1}\cA )-1| < \epsilon\},
\]
\[
r(\epsilon, h)= \frac{ N(\epsilon, h)}{\text{dim}(\cP_L^{-1}\cA)}.
\]
Here $N(\epsilon, h)$ counts the number of eigenvalues which lie within a circle of radius $\epsilon$ centered on real one while $r(\epsilon,h)$ gives the ratio of eigenvalues in the $\epsilon$-circle to the total number of eigenvalues. We evaluate $N(\epsilon, h)$ and $r(\epsilon, h)$ for $\epsilon \in \{ 0.2, 0.1, 0.05 \}$, $h=2^{-k}, \ k=2, \hdots, 6$, and as in Theorem \ref{th:distr2} we choose $\tau$ as $\tau^{2q-1}=h^2$. The results are shown in Table \ref{tab:numtest1}.
\begin{table}[H]
\caption{$N(\epsilon, h)$ and $r(\epsilon, h)$ for a range of $\epsilon$ and $h$ with $\tau^{2q-1}=h^2$ and q=3}
\label{tab:numtest1}
\centering
\begin{tabular}{||cc||cc|cc|cc||}
\hline
dim($\cP_L^{-1}\cA$) & $h=2^{-k}$ & $N(0.2)$ & $r(0.2)$  & $N(0.1)$ & $r(0.1)$& $N(0.05)$ & $r(0.05)$     \\
\hline
75 & 2 & 75 & 1.0 & 69 & 0.9200 & 59 & 0.7867 \\
243 & 3 & 243 & 1.0 & 240 & 0.9877 & 229 & 0.9424 \\
867 & 4 & 866 & 0.9988 & 861 & 0.9931 & 853 & 0.9839 \\
3267 & 5 & 3267 & 1.0 & 3259 & 0.9976 & 3246 & 0.9936 \\
12675 & 6 & 12675 & 1.0 & 12663 & 0.9991 & 12644 & 0.9976 \\
\hline
\end{tabular}
\end{table}
As predicted by Theorem \ref{th:distr2}, from Table \ref{tab:numtest1}, we clearly observe that as we decrease $h$ the percentage of eigenvalues that lie in the circles with radius $\epsilon$ centered at one increases monotonically and tends to $100\%$. This trend is seen across all examples except for $\epsilon=0.2$ and $h=2^{-4}$, where only a single eigenvalue falls outside of the circle.

\subsection{Test 2}
To illustrate the distributions predicted in Theorems \ref{th:distr1} and \ref{th:distr2} and in Remark \ref{other cases}, we define
\begin{align*}
E_1 &=SORT(|\lambda_j(\cP_L^{-1}\cA)|),
j=1,\ldots qn^2 \\
E_2 &=SORT([|1+\lambda_l(s(\theta_i,\theta_k))|, (1,...,1)]), \ i,j \in \{1,\hdots,n\}.
\end{align*}
Here $E_1$ denotes the sorted absolute values of the eigenvalues of the preconditioned system, $E_2$ denotes the sorted values of combination of two vectors, the first being a vector of all ones of length $n$, stemming from at least $n$ eigenvalues being one, see e.g., Theorem \ref{eigentructure-q=2} and \ref{eigentructure-q=3}. The other vector used in the construction has entries $|1+\lambda_l(s(\theta_i,\theta_k))|$ where $\lambda_l(s(\theta_i,\theta_k))$ are the zeros of the $q-1$ degree polynomials generated by $s$, given by \eqref{f-q=2} for the case $q=2$ and by the zeros of \eqref{lambda second degree} for $q=3$. Here $s(\theta_i,\theta_k)$ denotes the symbol of the matrix $\tau M^{-1} K$ sampled in $\theta_i,\theta_k$, it is given by
\[
s(\theta_i,\theta_k) = \frac{g_1(\theta_i,\theta_k)}{g_2(\theta_i,\theta_k)} , \  \theta_m=m\pi/(n+1) , \ m=1,\ldots,n,
\]
where $
g_1(\theta_1,\theta_2)= \frac{\tau}{3} (8 - 2\cos(\theta_1) - 2\cos(\theta_2)(1+2\cos(\theta_1)) )
$
 is the symbol for the stiffness matrix, including the $\tau$-scaling,  $g_2(\theta_1,\theta_2)=  \frac{4 h^2}{36}  ( 2 + \cos(\theta_1) )(2 + \cos(\theta_2) )$
is the symbol for the mass matrix $M$. We compute and plot $E_1$ and $E_2$ for $q=2,3$ for a range of discretizations and different choices of $\tau$. We expect that as we refine, the symbol of $\tau M^{-1} K$  better approximates the true eigenvalues of $\tau M^{-1} K$, and as a consequence, $E_1$ and $E_2$ should tend to superpose when $n\rightarrow \infty$. For each $q$ we choose $\tau$ in three different ways. More precisely, following the indications in Theorem \ref{th:distr2}, we choose $\tau$ such that $\tau^{2q-1}=h^2$. Then, in order to illustrate the statements in Remark \ref{other cases}, we choose $\tau=C h^2$ for $C=1,10$. The results for $q=2$ are shown in Figures \ref{q2matched},  \ref{q2c1}, \ref{q2c10} and for $q=3$ in Figures \ref{q3matched},  \ref{q3c1}, \ref{q3c10}.

\begin{figure}[htbp]
\centering
\includegraphics[width=0.32\textwidth]{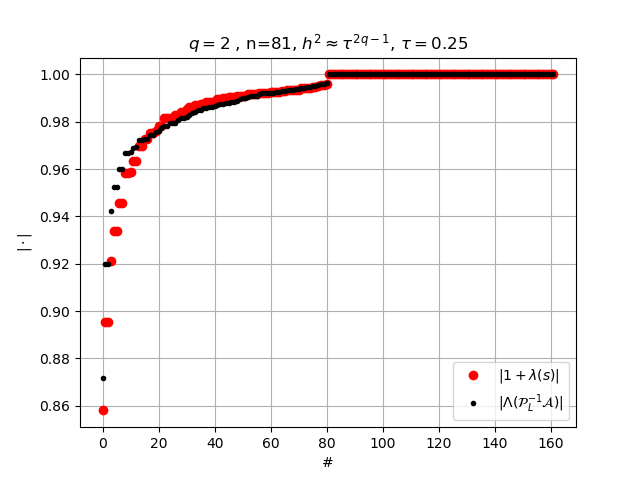}
\includegraphics[width=0.32\textwidth]{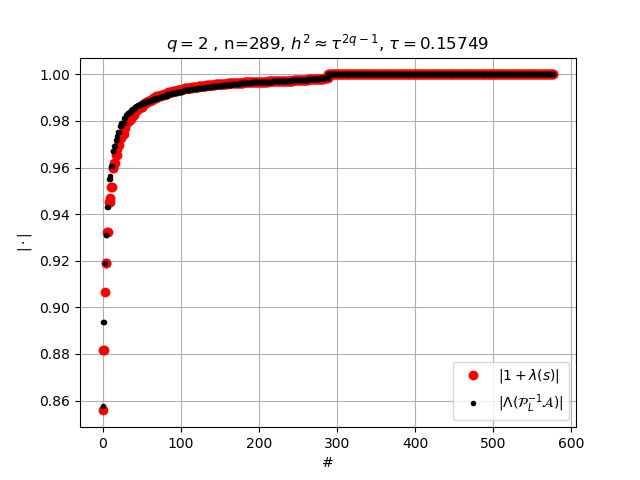}
\includegraphics[width=0.32\textwidth]{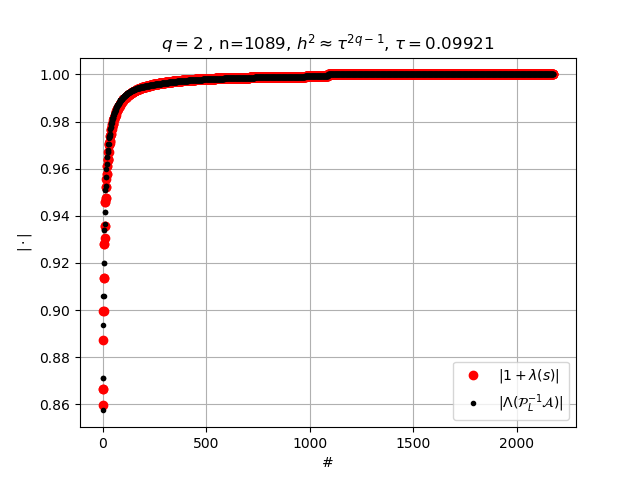}
\caption{$q=2$:  Symbol-prediction and eigenvalues of $\cP_L^{-1}\cA$ - range of $h$ with $\tau^{3}=h^2$.}
\label{q2matched}
\end{figure}

\begin{figure}[htbp]
\centering
\includegraphics[width=0.32\textwidth]{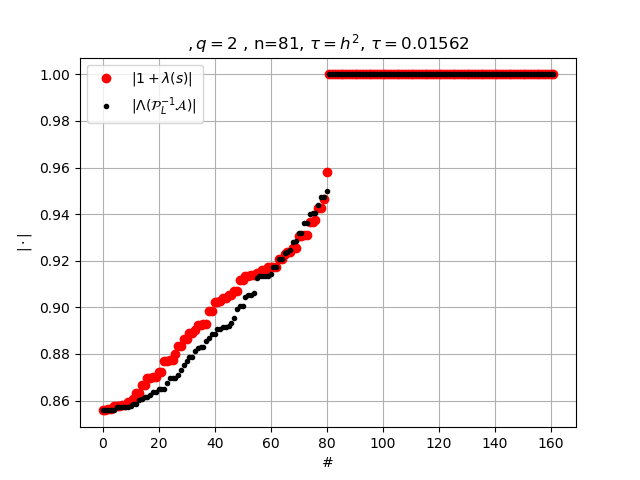}
\includegraphics[width=0.32\textwidth]{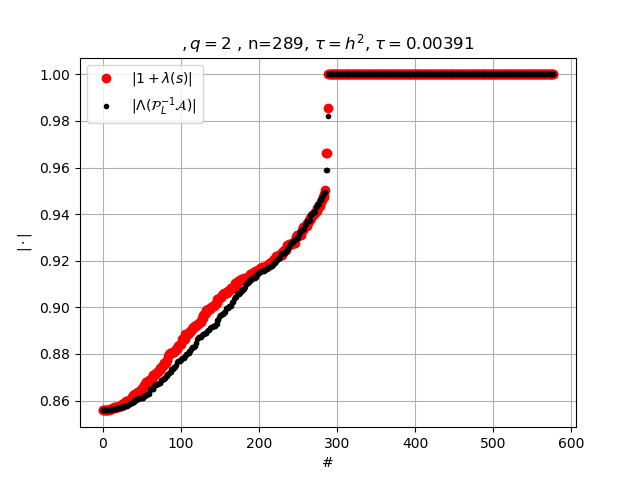}
\includegraphics[width=0.32\textwidth]{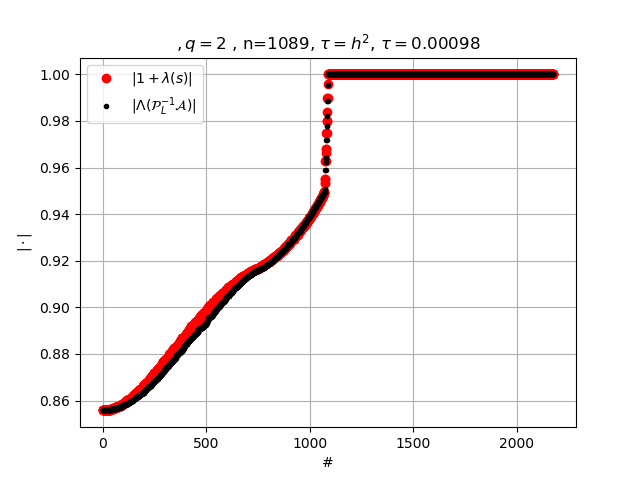}
\caption{$q=2$:  Symbol-prediction and eigenvalues of
$\cP_L^{-1}\cA$ - range of $h$ with $\tau=h^2$.}
\label{q2c1}
\end{figure}

\begin{figure}[htbp]
\centering
\includegraphics[width=0.32\textwidth]{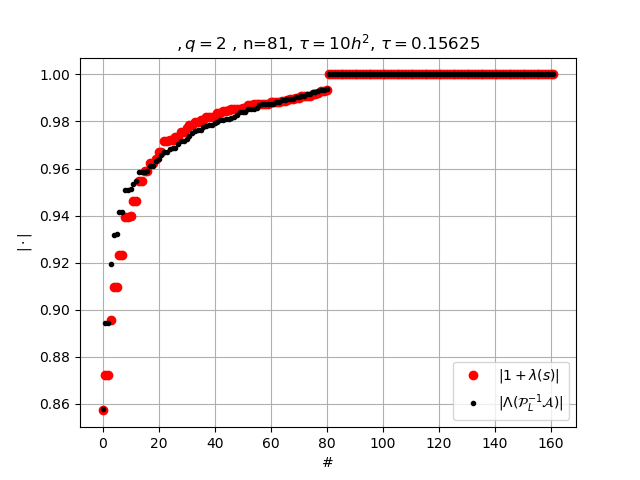}
\includegraphics[width=0.32\textwidth]{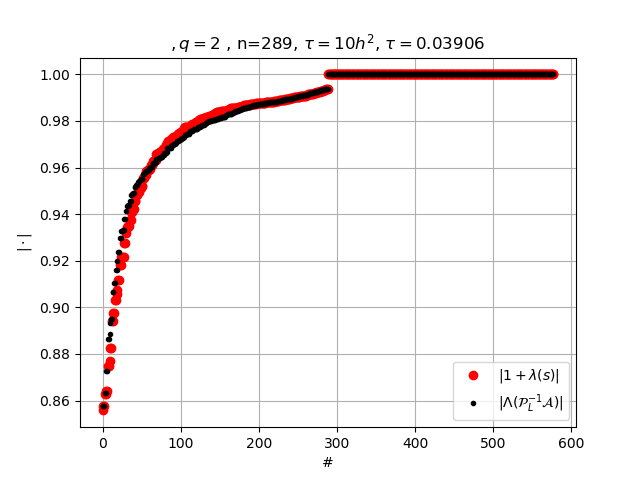}
\includegraphics[width=0.32\textwidth]{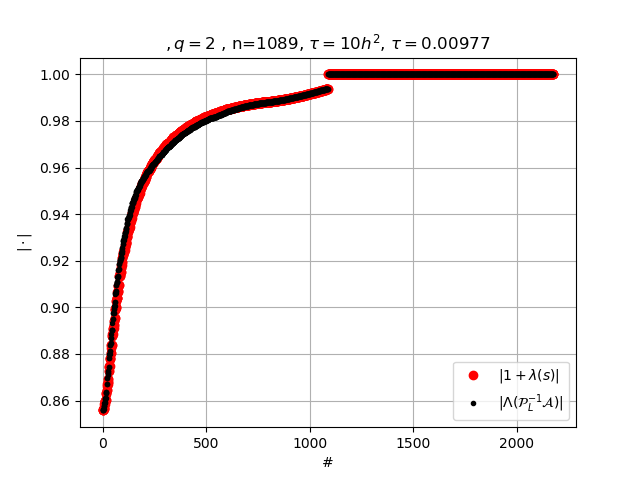}
\caption{$q=2$:  Symbol-prediction and eigenvalues of
$\cP_L^{-1}\cA$ - range of $h$ with $\tau=10h^2$.}
\label{q2c10}
\end{figure}

\begin{figure}[htbp]
\centering
\includegraphics[width=0.32\textwidth]{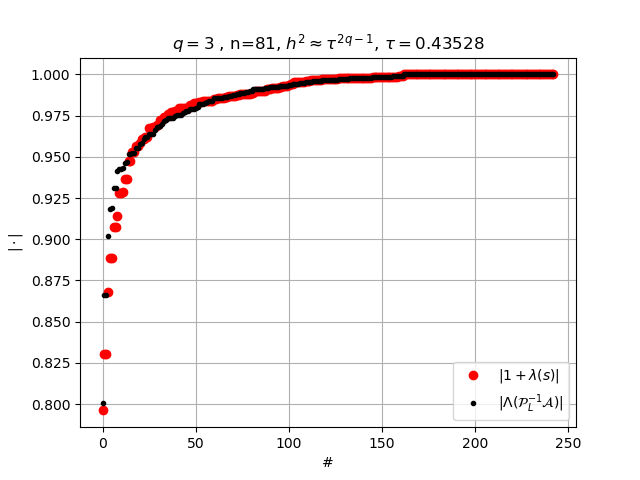}
\includegraphics[width=0.32\textwidth]{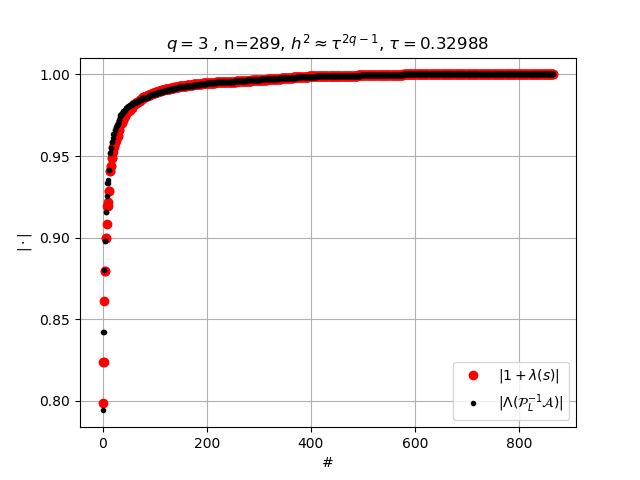}
\includegraphics[width=0.32\textwidth]{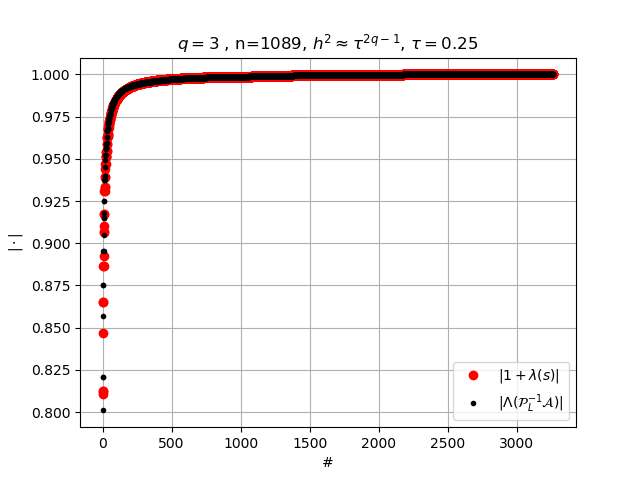}
\caption{$q=3$: Symbol-prediction and eigenvalues
$\cP_L^{-1}\cA$ - for a range of $h$ with $\tau^{5}=h^2$.}
\label{q3matched}
\end{figure}

\begin{figure}[htbp]
\centering
\includegraphics[width=0.32\textwidth]{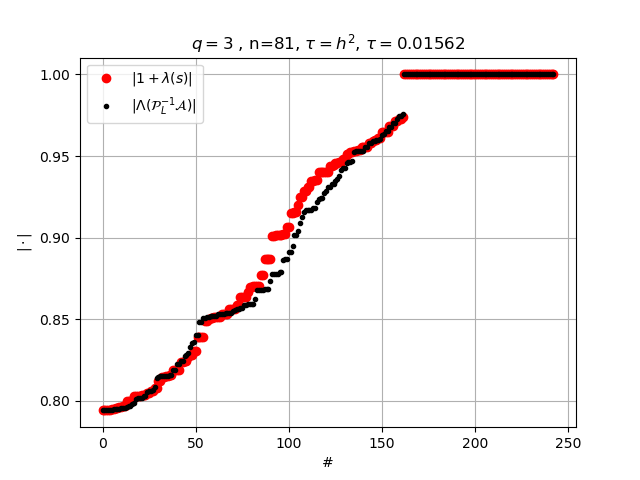}
\includegraphics[width=0.32\textwidth]{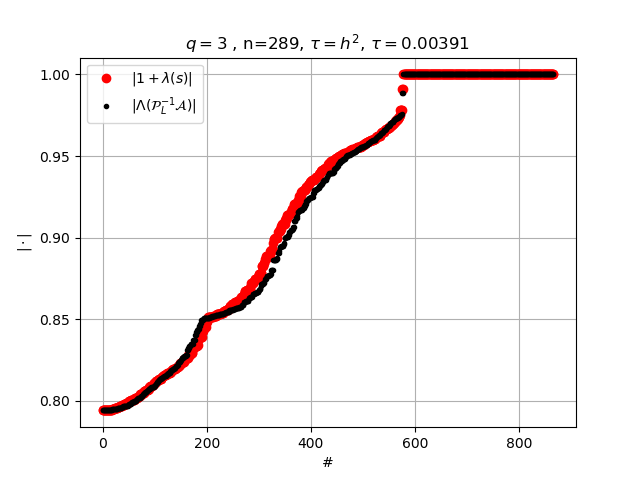}
\includegraphics[width=0.32\textwidth]{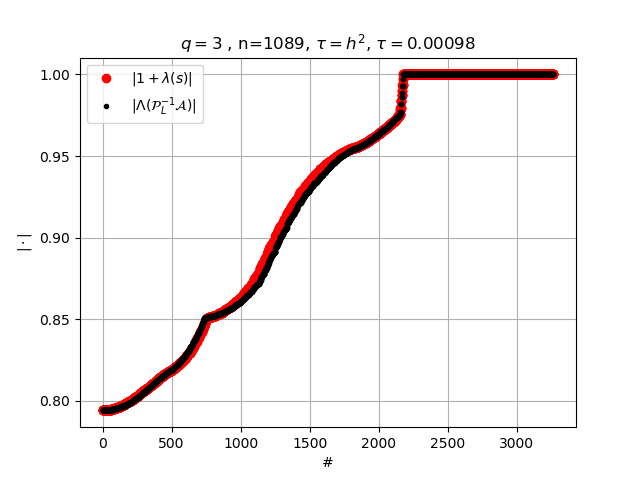}
\caption{$q=3$: Symbol-prediction and eigenvalues
$\cP_L^{-1}\cA$ - for a range of $h$ with $\tau=h^2$.}
\label{q3c1}
\end{figure}

\begin{figure}[htbp]
\centering
\includegraphics[width=0.32\textwidth]{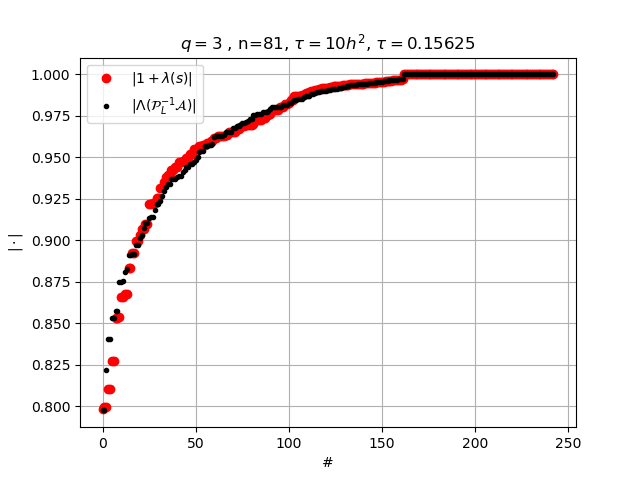}
\includegraphics[width=0.32\textwidth]{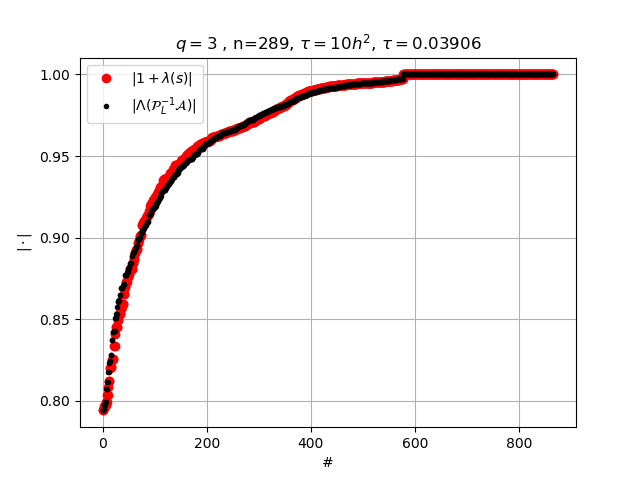}
\includegraphics[width=0.32\textwidth]{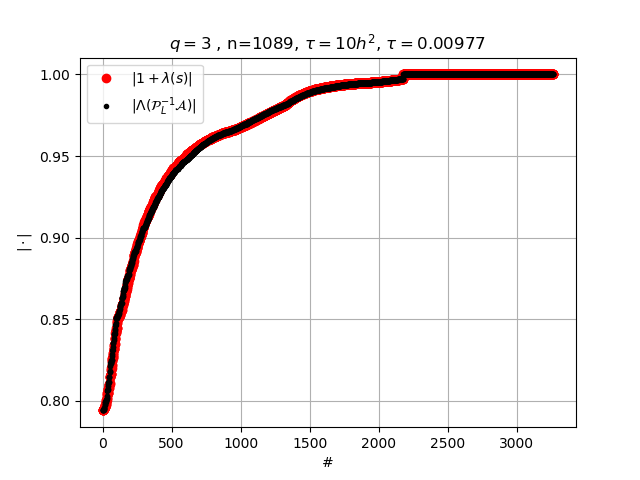}
\caption{$q=3$: Symbol-prediction and eigenvalues
$\cP_L^{-1}\cA$ - for a range of $h$ with $\tau=10h^2$.}
\label{q3c10}
\end{figure}

In all cases we observe in a convincing way a complete adherence of the numerical results with the theoretical findings. Indeed, the eigenvalues follow the prediction indicated by the symbol, and the match improves with the refinement. Finally, the clustering property shown in Table \ref{tab:numtest1} in the current section is illustrated in Figure \ref{q3matched}.

The numerical experiments in this work treat the distribution of the eigenvalues of the preconditioned matrices, preconditioner performance when solving the linear systems is documented in \cite{IRKtheory} and \cite{stageparallel} for range of $q\in\{1,\hdots,9\}$.

\section{Conclusions}\label{sec:final}
In this work we consider strongly A-stable implicit Runge-Kutta methods of arbitrary order of accuracy, for which an efficient preconditioner has been introduced. We present a refined spectral analysis of the corresponding matrices and matrix-sequences, both in terms of localization,  asymptotic global distribution, and explicit expressions of the eigenvectors, by using matrix theoretical and spectral tools reported in Section \ref{sec:sp_tools}. The presented theoretical analysis fully agrees with the numerically observed spectral behavior and substantially improves the theoretical study done in this direction so far.
A wide set of numerical experiments is included and critically discussed.
As future steps there are many more intricate cases that can be treated let us say directly by our very parametric theoretical setting.
\begin{itemize}
\item The case where the IRK method is considered with higher order in space approximation leads to the case of multilevel Toeplitz matrices, say $M,K$, and preconditioned Toeplitz matrices $M^{-1}K$ having a matrix-valued symbol that is with $r>1$ according to the notation in Section \ref{ssec:toeplitz-distribution}. In that setting the equivalent of the matrix $Z_\tau$ has a spectral behavior for fixed dimension and asymptotically which is known in detail; the same is true when we consider a nonsymmetric stiffness matrix i.e. when the term $\bm{b}$ in \rref{eq_pde} is nonzero (see \cite{GSS,GLT-blockdD}).
\item Since the approach is very general as emphasized in the parametric derivation in Section \ref{q=3}, other discretization methods in time that give raise to a different Butcher tableau could be analysed in a similar manner.
\item The current type of analysis is reminiscent of the bifurcation theory, which has been used in several settings in pure and applied mathematics, as well in science and engineering contexts (see e.g. \cite{bif1,bif2,bif3,bif4,bif5} and references therein). A direction to investigate is the use of such tools for deriving properties of the branches and expressions or bounds for $r_q^*$ in a rigorous way, especially when the Runge-Kutta parameter $q$ is either moderate or large.
\item Finally, we believe that the case of more general domain and variable coefficients could be a challenge also for the GLT theory \cite{GLT1_book,GLT2_book,Barb,GLT-block1D, GLT-blockdD}, given the intrinsic nonsymmetric nature of the original coefficient matrix and of the resulting preconditioning, even if the analysis of $M,K$, $M^{-1}K$ is easily available also for $r>1$.
\end{itemize}

\end{document}